\documentclass{amsart}
\usepackage{mathrsfs}
\usepackage{fullpage}
\usepackage{caption}
\usepackage{latexsym,amsfonts,amssymb,epsfig,verbatim}
\usepackage{amsmath, latexsym, graphics, textcomp}
\usepackage{pinlabel}

\usepackage{url}
\usepackage{color}
\usepackage[colorlinks,citecolor=blue]{hyperref}
\usepackage{accents}
\usepackage{graphicx,enumerate}
\usepackage{tikz}
\usetikzlibrary{matrix,arrows,decorations.pathmorphing}
\usetikzlibrary{cd}

\theoremstyle{definition}
\newtheorem{theorem}[equation]{Theorem}
\newtheorem{proposition}[equation]{Proposition}

\newtheorem{corollary}[equation]{Corollary}
\newtheorem{lemma}[equation]{Lemma}
\newtheorem{claim}[equation]{Claim}
\newtheorem{definition}[equation]{Definition}

\newtheorem{remark}[equation]{Remark}

\newtheorem{question}[equation]{Question}

\numberwithin{equation}{section}

\newcommand{\set}[1]{\{ #1 \}}
\newcommand{\gen}[1]{\langle #1 \rangle}
\newcommand{\mcg}[1]{\text{Mod}(#1)}
\newcommand{\ZZ}{\mathbb{Z}}
\newcommand{\HH}{\mathbb{H}}
\newcommand{\cc}[1]{\mathcal{C}(#1)}
\newcommand{\ccs}[1]{\mathcal{C}^s(#1)}
\newcommand{\acc}[1]{\mathcal{AC}(#1)}
\newcommand{\ac}[1]{\mathcal{A}(#1)}
\newcommand{\hh}{\mathfrak{H}}
\newcommand{\ssp}[2]{\pi_{#1}(#2)}
\newcommand{\dec}[1]{\Delta_{#1}}

\newcommand{\frakh}{\mathfrak{h}}
\newcommand{\calH}{\mathcal{H}}
\newcommand{\calM}{\mathcal{M}}
\newcommand{\calN}{\mathcal{N}}
\newcommand{\calR}{\mathcal{R}}
\newcommand{\calS}{\mathcal{S}}
\newcommand{\zeroto}[1]{\{0,\ldots,#1\}}
\newcommand{\oneto}[1]{\{1,\ldots,#1\}}
\newcommand{\II}{\mathbb{I}}
\DeclareMathOperator{\id}{id}
\DeclareMathOperator{\diam}{diam}
\DeclareMathOperator{\len}{length}

\newcommand{\Aut}{\mathrm{Aut}}

\newcommand{\Homeo}{\mathrm{Homeo}}

\newcommand{\Stab}{\mathrm{Stab}}

\title{Pseudo-Anosov subgroups of surface bundles over tori}
\author{Junmo Ryang}
\address{Department of Mathematics, Rice University, Houston, TX}
\email{jr95@rice.edu}
\thanks{The author was supported in part by the NSF grant DMS-1745670.}

\begin{document}

\begin{abstract}
	We show that finitely generated, purely pseudo-Anosov subgroups of the fundamental groups of surface bundles over tori are convex cocompact as subgroups of the mapping class group via the Birman exact sequence. This generalizes the fact that similar groups within fibered 3-manifold groups are convex cocompact, which is a combination of results due to Dowdall, Kent, Leininger, Russell, and Schleimer.
\end{abstract}

\maketitle

\section{Introduction}
Farb and Mosher first defined convex cocompactness in mapping class groups \cite{MR1914566}, and this property has since been studied from a variety of perspectives \cite{hamenstadt2005word,MR2349677,MR2465691,MR2437226,MR3000500,MR3426695,MR4069890}. Convex cocompact subgroups must necessarily be finitely generated and purely pseudo-Anosov, but a major open question is whether the converse holds \cite[Question 1.5]{MR1914566}. If the converse does hold, then work of Farb-Mosher \cite{MR1914566} and Hamenst\"{a}dt \cite{hamenstadt2005word} implies that the fundamental groups of all compact atoroidal surface bundles are word hyperbolic; see also \cite{MR3000500}. If the converse does not hold, then one obtains some interesting examples of finitely generated groups with no Baumslag-Solitar subgroups that fail to be hyperbolic; see \cite[Section 8]{MR2342811}, \cite[Question 1.1]{bestvinaquestions}, \cite{MR1724853}, \cite{MR4526820}. There exist several partial results which show the converse does hold when restricted to particular subgroups of the mapping class group \cite{MR2599078,MR3314946,MR3695858,MR4308279,MR4632569,MR4684593,chesser2023purely}. In particular, the converse holds within fibered 3-manifold groups embedded in the mapping class group of a punctured surface via the Birman exact sequence \cite{MR2599078,MR3314946,MR4632569}. In the current work, we generalize the result to the fundamental groups of surface bundles fibering over an $n$-torus, where the $n=1$ case is precisely the fibered $3$-manifold case.

We give a more precise description. Let $S$ be a connected, orientable, finite-type surface of negative Euler characteristic, and let $E$ be a fiber bundle over $T^n$ with fiber $S$. Given a point $z \in S$, we write $S^z = S \setminus \set{z}$. The short exact sequence of fundamental groups associated to the fiber bundle $E$ maps into the Birman exact sequence \cite{MR0243519} via the monodromy representation $\mu$,

\[\begin{tikzcd}
    1 \ar{r} & \pi_1S \ar{r} \ar{d}{=} & \pi_1E \ar{r} \ar{d}{\mu^z} & \ZZ^n \ar{r} \ar{d}{\mu} & 1\phantom{.} \\
    1 \ar{r} & \pi_1S \ar{r} & \mcg{S^z,z} \ar{r}{\Phi_*} & \mcg{S} \ar{r} & 1.
\end{tikzcd}\]

\noindent See \cite[Section 1.2]{MR1914566}.

Let $\Gamma$ denote the image of $\pi_1E$ in $\mcg{S^z,z}$, the subgroup of $\mcg{S^z}$ consisting of mapping classes fixing the $z$-puncture. We now state the main theorem.
\begin{theorem}\label{thm:main}
    Suppose $\chi(S) < 0$ and $E$ is an $S$-bundle over $T^n$. A subgroup $G < \Gamma = \mu^z(\pi_1E)$ is convex cocompact if and only if it is finitely generated and purely pseudo-Anosov.
\end{theorem}

\subsection{Surface group extensions}
Although Theorem \ref{thm:main} is stated from a geometric point of view about surface bundles $E$, the content of the result applies most naturally to surface group extensions $\Gamma$ in punctured mapping class groups.

Given a subgroup $H < \mcg{S}$, let $\Gamma_H$ denote its full preimage under the puncture forgetting map $\Phi_*: \mcg{S^z,z} \to \mcg{S}$. This $\Gamma_H$ is a $\pi_1S$-extension which embeds naturally into the Birman exact sequence; as above.

\[
\begin{tikzcd}
    1 \ar{r} & \pi_1S \ar{r} \ar[phantom]{d}[marking]{=} & \Gamma_H \ar{r} \ar[phantom]{d}[marking]{<} & H \ar{r} \ar[phantom]{d}[marking]{<} & 1 \\
    1 \ar{r} & \pi_1S \ar{r} & \mcg{S^z,z} \ar{r}{\Phi_*} & \mcg{S} \ar{r} & 1
\end{tikzcd}
\]

When $H$ is an infinite cyclic group generated by a pseudo-Anosov element, Thurston's hyperbolization theorem gives that $\Gamma_H$ is isomorphic to the fundamental group of a finite volume, hyperbolic, fibered 3-manifold. In this setting, \cite{MR3314946} show that finitely generated, purely pseudo-Anosov subgroups $G < \Gamma_H$ are convex cocompact. If $H$ is instead generated by an infinite order reducible element, $\Gamma_H$ is then the fundamental group of a closed, non-hyperbolic, fibered 3-manifold. Despite this difference, \cite{MR4632569} show that finitely generated, purely pseudo-Anosov subgroups are convex cocompact in this setting as well. The case that $H$ is finite follows from \cite{MR2599078}. The following natural question then arises.

\begin{question} \label{q:main}
    For which groups $H$ do the extensions $\Gamma_H$ have the property that all finitely generated, purely pseudo-Anosov subgroups are convex cocompact?
\end{question}

We colloquially refer to this property as `the converse property'. As stated above, \cite{MR4632569} concluded that surface-by-cyclic $\Gamma_H$ have the converse property. We prove that surface-by-free-abelian extensions also have the converse property.

\begin{theorem} \label{thm:main2}
    Suppose $\chi(S) < 0$ and $H$ is a free abelian subgroup of $\mcg{S}$. A subgroup $G < \Gamma_H$ is convex cocompact if and only if it is finitely generated and purely pseudo-Anosov.
\end{theorem}

Given Theorem \ref{thm:main2}, finite index considerations allow us to immediately extend the result to surface-by-abelian and surface-by-solvable extensions.

\begin{lemma} \label{lem:finiteindex}
    Suppose $\chi(S) < 0$, $H$ is a subgroup of $\mcg{S}$, and $H' < H$ is a finite index subgroup. The extension group $\Gamma_H$ has the converse property if and only if $\Gamma_{H'}$ has the converse property.
\end{lemma}
\begin{proof}
    Since $\Gamma_{H'} < \Gamma_{H}$, one direction of implication is obvious. For the other direction, we suppose that $\Gamma_{H'}$ has the converse property and then show that $\Gamma_{H}$ has it as well.
    
    Let $G < \Gamma_{H}$ be a finitely generated, purely pseudo-Anosov subgroup. Using the curve complex characterization of convex cocompactness, it is easy to see that $G$ is convex cocompact if and only if any finite index subgroup is convex cocompact; see Section \ref{sec:ccc}. Consider the subgroup $G' = G \cap \Gamma_{H'}$. Since $H' < H$ is finite index, we know $\Gamma_{H'} < \Gamma_{H}$ is finite index, and so $G' < G$ is finite index. Since $\Gamma_{H'}$ has the converse property, $G'$ is convex cocompact. Thus, $G$ is convex cocompact, as desired.
\end{proof}

\begin{corollary} \label{cor:main3}
    Suppose $\chi(S) < 0$ and $H$ is an abelian (or solvable) subgroup of $\mcg{S}$. A subgroup $G < \Gamma_H$ is convex cocompact if and only if it is finitely generated and purely pseudo-Anosov.
\end{corollary}
\begin{proof}
    The results of \cite{MR0726319} and \cite{MR1195787} assert that every solvable group contains a finite index abelian subgroup, and moreover every abelian subgroup is finitely generated. Thus, if $H$ is solvable or abelian, then $H$ must contain a free abelian subgroup $H'$ of finite index. Theorem \ref{thm:main2} states that $\Gamma_{H'}$ has the converse property, so Lemma \ref{lem:finiteindex} gives that $\Gamma_{H}$ also has the converse property.
\end{proof}

Theorem \ref{thm:main} follows from Theorem \ref{thm:main2} and Corollary \ref{cor:main3}. To see the connection, note that $S$-bundles over an $n$-torus have surface-by-free-abelian fundamental groups. If such a bundle $E$ has injective monodromy representation, then $\pi_1E$ embeds into $\mcg{S^z,z}$ and its image $\Gamma$ has the converse property by Theorem \ref{thm:main2}. If $E$ instead has non-injective monodromy, then $\pi_1E$ maps with some nontrivial kernel to a surface-by-abelian $\Gamma < \mcg{S^z,z}$, which has the converse property by Corollary \ref{cor:main3}. For the remainder of the paper, we focus on proving Theorem \ref{thm:main2}.

\subsection{Proof Summary} \label{sec:proofsummary}
To prove Theorem \ref{thm:main2}, we fix a finitely generated, purely pseudo-Anosov subgroup $G < \Gamma_H$, and show that the orbit map of $G$ to the curve complex $\cc{S^z}$ is a quasi-isometric embedding. From \cite{MR2349677,MR2465691}, this is equivalent to $G$ being convex cocompact; see Theorem \ref{thm:qi-ccc}. So, the central task is to find a way to relate distances in $G$ to distances in $\cc{S^z}$. The works of \cite{MR2599078,MR3314946,MR4632569} provide examples of successful approaches to this kind of problem. These approaches also answer subcases of our question when $H$ has low rank.

Let $n$ be the rank of the free abelian group $H$. If $n=0$, $\Gamma_H$ is exactly the Birman kernel $\pi_1S$, so we recover the setting of \cite{MR2599078}, in which $G < \pi_1S$. In this setting, the authors of \cite{MR2599078} relate distances in $G$ to distances in $\cc{S^z}$ by examining $K_u$, the stabilizer in $\pi_1S < \mcg{S^z}$ of a simplex $u \subset \cc{S^z}$. Using the isometric action of $\pi_1S$ by deck transformations on the universal cover $p: \HH^2 \to S$, we define $\hh_u$ to be the convex hull of the limit set of $K_u$ in $\partial \HH^2$. Because $G < \pi_1S$, $G$ also has a convex hull $\hh_G$. Since $G$ acts on $\hh_G$ geometrically, $\hh_G$ serves as a geometric model for $G$. A key result in \cite{MR2599078} states that hull intersections $\hh_u \cap \hh_G$ have uniformly bounded diameter, independent of $u$. The simplices that make up a geodesic edge path between $G$-orbit points in $\cc{S^z}$ then give rise to a chain of bounded diameter sets in $\hh_G$. The total diameter of this chain bounds distance as a linear function of the distance in $\cc{S^z}$, as required for the quasi-isometric condition.

If instead $n=1$ with $H$ generated by a pseudo-Anosov mapping class, we recover the setting of \cite{MR3314946}, in which $\Gamma_H$ is isomorphic to the fundamental group of a hyperbolic $3$-manifold. The authors of \cite{MR3314946} adopt a similar approach to \cite{MR2599078} by taking convex hulls of $K_u$ and $G$ in $\HH^3$ rather than $\HH^2$. Again, the key result is that hull intersections $\hh_u \cap \hh_G$ are uniformly bounded. The same argument using a chain of bounded diameter sets in $\hh_G$ gives the quasi-isometric condition.

If again $n=1$ but $H$ is generated by a reducible mapping class, we recover the setting of \cite{MR4632569}, in which $\Gamma_H$ is isomorphic to the fundamental group of a non-hyperbolic $3$-manifold. In this setting, the authors of \cite{MR4632569} again find appropriate objects to take the place of the universal cover $\HH^2$ and convex hulls $\hh_u$ and $\hh_G$, and continue with a similar approach. Because $H$ is reducible, it has an associated canonical reduction system $\alpha$ on $S$. In place of $\HH^2$, \cite{MR4632569} use the tree $T$ dual to the lifts of $\alpha$ in $\HH^2$. Since $H$ is cyclic, it can be realized by homeomorphisms, so $\Gamma_H$ splits as a semi-direct product $\pi_1S \rtimes \gen{f}$ where $f$ is the representative of the generator of $H$. The isometric action of $\pi_1S$ on $\HH^2$ thus extends to a non-isometric action of $\Gamma_H$ by lifting $f$ to $\widetilde{f}:\HH^2 \to \HH^2$, which induces an isometric action of $\Gamma_H$ on the dual tree $T$. In place of convex hulls $\hh_u$ and $\hh_G$, \cite{MR4632569} use $K_u$-invariant subtrees $T^u \subset T$ and a $G$-invariant subtree $T^G \subset T$. Being purely pseudo-Anosov implies that $G$ acts freely on $T$, and a short argument proves $T^G$ is a geometric model of $G$. The key result is now that tree intersections $T^u \cap T^G$ are uniformly bounded, independent of $u$. The simplices of a geodesic edge path in $\cc{S^z}$ now correspond to a chain of bounded diameter sets in $T^G$, and so the quasi-isometric condition follows.

Now consider the remaining case where $n \geq 2$. If $H$ were irreducible, it would contain a pseudo-Anosov element whose centralizer in $\mcg{S}$ must contain $H$. This contradicts the fact that the centralizer of a pseudo-Anosov is virtually cyclic \cite{mccarthy1982normcent}. Thus, $H$ must necessarily be a reducible subgroup of $\mcg{S}$. Further, $H$ can also be realized by homeomorphisms; see Section \ref{sec:realizable}. Since the group $H$ in our setting maintains these key characteristics as in the setting of \cite{MR4632569}, we begin with a similar approach using the dual tree $T$ and the subtrees $T^u$ and $T^G$. However, the increased rank of $H$ demands that our argument diverge when proving the key result that tree intersections $T^u \cap T^G$ are uniformly bounded.

\subsubsection{Bounding $T^u \cap T^G$}
To understand $T^u \cap T^G$, we return to examining convex hulls. Although the action of $\Gamma_H$ on $\HH^2$ is non-isometric, it still induces an action on $\partial \HH^2$ by homeomorphisms. Thus, we again define the convex hull $\hh_G$, and note that $\hh_G$ admits an isometric action by $G_0 = G \cap \pi_1S$. Further, there is a $G_0$-equivariant inclusion $T^G \to \hh_G$, since $G_0$ acts freely on $T^G$. The quotient $p_0: \hh_G \to \hh_G / G_0 = \Sigma_0$ is an infinite-type surface, and $T^G / G_0 = \sigma_0$ is a spine. The surface $\Sigma_0$ admits a cocompact, non-isometric action by the quotient group $G/G_0$. While the action of $G/G_0$ on $\Sigma_0$ is not isometric, the induced action on the spine $\sigma_0$ is isometric; see Section \ref{sec:Gquotient}. In the $n=1$ case, $G/G_0 \cong H \cong \ZZ$, so $\Sigma_0$ admitting a cocompact $\ZZ$-action showed it was two-ended. In the $n \geq 2$ case, $G/G_0 \cong H \cong \ZZ^n$, so $\Sigma_0$ is one-ended.

To bound $T^u \cap T^G$ in the $n=1$ case, \cite{MR4632569} construct a compact subsurface $\Sigma_1 \subset \Sigma_0$, and focus on those simplices $u \subset \cc{S^z}$ with the property that $p_0(T^u \cap T^G) \subset \Sigma_1$. A bound on the tree intersections of these `deep' simplices extends to a bound on all simplices by leveraging the action of $G/G_0$. After reducing to deep simplices, \cite{MR4632569} divide $T^u \cap T^G$ into two subsets and bound each separately. The first subset, called the `hull subtree', is spanned by vertices dual to regions intersecting $\hh_u \cap \hh_G$. The hull subtree is bounded by appealing to \cite{MR2599078}, after noting that the compactness of $\Sigma_1$ implies $G_1$ is finitely generated. The second subset consists of components called `parallel subtrees'. It is shown that a long geodesic in a parallel subtree gives rise to both a long segment of a simple closed geodesic in $S$ and a long segment of a filling geodesic in $S$, and these two long segments must run parallel to each other. If the parallel segments are too long, then this is a contradiction, so geodesics in the parallel subtrees must necessarily be short, furnishing the bound.

In the $n \geq 2$ case, we mimic the 3-step process of 1) constructing $\Sigma_1$, 2) reducing to deep simplices, and 3) bounding the hull and parallel subtrees. The main difficulty lies in the first step. If $\Sigma_1$ is to be compact, we must show that $p_0(T^u \cap T^G)$ is uniformly bounded, independent of $u$. The corresponding argument in the $n=1$ case relies on the two-ended shape of $\Sigma_0$, so we need a new argument for the $n \geq 2$ case, which we outline in the following subsection. The second step of reducing to deep simplices only requires minor updates to accommodate the change from a $\ZZ$-action to a $\ZZ^n$-action on $\Sigma_0$. The third step requires no change, as the arguments in the $n=1$ case are agnostic to the rank of $H$ and the shape of $\Sigma_0$. See Section \ref{sec:secondreduction} for more details.

\subsubsection{Bounding $p_0(T^u \cap T^G)$}
To bound $p_0(T^u \cap T^G)$, the main tool used is the subsurface projection machinery of Masur and Minsky \cite{MR1791145}. For simplicity, we describe the idea in the case where $H$ is rank $2$, generated by two right-handed Dehn twists $h_1,h_2$ about disjoint, non-isotopic, simple, closed curves $\alpha_1,\alpha_2$.

Let $A_1, A_2 \to S$ be the annular covers whose core curve are $\alpha_1,\alpha_2$, respectively. For every simplex $u \subset \cc{S^z}$, let $\ssp{j}{v}$ be the subsurface projection of $v = \Phi(u) \subset \cc{S}$ to the arc graph $\ac{A_j}$ of the annular cover $A_j$; see Section \ref{sec:ssp}. For every edge $e \subset T^G$, there is a dual geodesic $\widetilde{\alpha}_e \subset p^{-1}(\alpha_j)$ for either $j=1 \text{ or } 2$, and we identify the annulus $A_j$ with the quotient $A_e = \HH^2 / \Stab_{\pi_1S}(\widetilde{\alpha}_e)$. There are two boundary components of $\partial \hh_G$ that non-trivially intersect $\widetilde{\alpha}_e$, and we let $\dec{e}$ denote their image in $A_j$, viewed as a subset of the arc graph $\ac{A_j}$; see Section \ref{sec:decorations}. These decorations $\dec{e}$ associated to edges $e \subset T^G$ are $G$-equivariant; see Lemma \ref{lem:decaction}.

An important fact proved by Leininger-Russell is that any edge $e$ with large distance between $\dec{e}$ and $\ssp{j}{v}$ in $\ac{A_j}$ cannot lie in the interior of a hull or parallel subtree. Such edges must serve as `dead ends' at the leaves of the hull or parallel subtrees or else lie outside of $T^u \cap T^G$; see Lemmas \ref{lem:Bparallel} and \ref{lem:Bhull}, or \cite[Lemmas 5.6 and 5.9]{MR4632569} for details. The proof of this fact is independent of the rank of $H$, so the fact remains true in our case. As a consequence, interior edges of $T^u \cap T^G$ have small distance between $\dec{e}$ and $\ssp{j}{v}$, with the exception of those edges for which $\ssp{j}{v} = \emptyset$. We collect the edges with small $d(\dec{e},\ssp{j}{v})$ in a set $\widetilde{\mathcal{E}} \subset T^G$ and consider its $p_0$-image $\mathcal{E} \subset \sigma_0$.

In the $n=1$ case, Leininger and Russell show that $\mathcal{E}$ is uniformly bounded, independent of $v$. Since parallel subtrees contain no edges with empty subsurface projection, this gives a bound on the $p_0$-image of parallel subtrees. The hull subtree may contain `gaps' of edges with empty subsurface projection, but these gaps can also be uniformly bounded to give a bound on the $p_0$-image of the hull subtree. Combining these bounds results in a bound on $p_0(T^u \cap T^G)$. See \cite[Section 5]{MR4632569}.

Returning to our case, $\mathcal{E}$ is no longer a bounded set, so the above approach cannot work. Instead, we partition $\mathcal{E}$ into the subsets $\mathcal{E}_1, \mathcal{E}_2$ which contain only edges dual to $\alpha_1, \alpha_2$, respectively. Further, we define subgroups $H_1,H_2 < H \cong G/G_0$ which act coarsely as the identity on $\ac{A_1},\ac{A_2}$, respectively, and we define quotient graphs $\sigma_1 = \sigma_0 / H_1$ and $\sigma_2 = \sigma_0 / H_2$ with quotient maps $p_1: \sigma_0 \to \sigma_1$ and $p_2: \sigma_0 \to \sigma_2$. While our partition sets $\mathcal{E}_j \subset \sigma_0$ are still unbounded, their images $p_j(\mathcal{E}_j) \subset \sigma_j$ are bounded. The bounds are obtained by leveraging the loxodromic actions of elements in $H/H_j$ on $\ac{A_j}$; see Lemma \ref{lem:pjkbounded}. Moreover, $\sigma_0$ is quasi-isometric to the product of graphs $\sigma_1 \times \sigma_2$; see Lemma \ref{lem:qitosigma0}. The final ingredient is the fact that geodesics in $T^G$ must periodically visit both $\alpha_1$-dual edges and $\alpha_2$-dual edges; see Corollary \ref{cor:Lbound}. Thus, geodesics in the $p_0$-image of parallel subtrees must be contained in a bounded neighborhood of the intersection $\mathcal{E}_1 \cap \mathcal{E}_2$. The image of the intersection in $\sigma_1 \times \sigma_2$ is bounded, and so $\mathcal{E}_1 \cap \mathcal{E}_2$ is bounded in $\sigma_0$. Geodesics in the hull subtree can be similarly bounded after an additional argument bounding the `gaps' of edges with empty subsurface projection. Combining the hull and parallel subtree bounds gives a bound on $p_0(T^u \cap T^G)$.

Our proof in the general case for $H$ with arbitrary rank and arbitrary reducible generators follows the same basic idea, using subsurface projections to complementary components $S \setminus \alpha$ to build decorations on vertices of $T^G$ in addition to the decorations on edges coming from the annular covers. The argument is complicated by the fact that arbitrary generators may not obviously align with a subsurface on which it acts loxodromically; see Section \ref{sec:main} for details.

\subsection{Acknowledgments}
The author would like to express great and heartfelt thanks to Chris Leininger and Jacob Russell for many enlightening conversations and support throughout the work. The author would also like to thank George Domat, Khanh Le, and Brian Udall for helpful conversations. Finally, the author would like to give general thanks to the fellow graduate students at Rice University for providing community and support.

\section{Preliminaries}
Let $S$ be a connected, orientable, finite-type surface with $\chi(S) < 0$. Fix a complete hyperbolic metric of finite area on $S$ that identifies $\HH^2$ with the universal cover $p: \HH^2 \to S$. Given a point $z \in S$, let $S^z$ denote the surface obtained by puncturing $S$ at $z$. The surface $S^z$ similarly admits a complete hyperbolic metric of finite area.

\subsection{Curve complexes and arc complexes}
The \emph{curve complex of $S$} is the flag simplicial complex $\cc{S}$ whose vertices are isotopy classes of essential, simple, closed curves on $S$ with two isotopy classes joined by an edge if they have disjoint representatives. Each vertex of $\cc{S}$ has a unique geodesic representative, and two vertices will be joined by an edge if and only if these geodesic representatives are disjoint. Hence, each simplex of $\cc{S}$ corresponds to a multicurve on $S$, which has a unique geodesic representative. Whenever convenient, we will assume that a simplex/multicurve $v \subset \cc{S}$ is represented in $S$ as a geodesic multicurve.

Given a surface with boundary $Y$, the \emph{arc and curve complex of $Y$} is the flag simplicial complex $\acc{Y}$ whose vertices are isotopy classes of both essential, simple, closed curves on $Y$ and essential arcs on $Y$ meeting the boundary $\partial Y$ precisely at their endpoints. As with the curve complex, two vertices of $\acc{Y}$ are joined by an edge if there are disjoint representatives for the isotopy classes.

When $S$ is a once-punctured torus or four-punctured sphere, one usually makes an alternate definition of $\cc{S}$, but we \underline{do not} do that here. In particular, we maintain that these curve complexes are discrete, countable sets. On the other hand, if $Y$ is a torus with one boundary component or a sphere with at least one boundary component and the sum of the boundary components and punctures equal $4$, then we \underline{do} take the usual alternate definition of $\acc{Y}$: vertices are now joined by an edge if they intersect once or twice (rather than zero times) for these two types of surfaces with boundary, respectively. The reason is that for $\cc{S}$, we need Theorem \ref{thm:tree-curve-complex} to hold, while for $\acc{Y}$, we will use coarse geometric properties in Section \ref{sec:main}.

If $A \to S$ is an annular cover, let $\overline{A}$ denote the compact annulus obtained from $A$ by adding its ideal boundary from the hyperbolic metric on $S$. This compactification is independent of the choice of metric. The \emph{arc complex} $\ac{A}$ is the flag simplicial complex whose vertices are isotopy classes of essential arcs on $\overline{A}$, where unlike other surfaces with boundary, isotopies of $\overline{A}$ are required to be the identity on $\partial \overline{A}$. Edges of $\ac{A}$ correspond to pairs of isotopy classes with representatives having disjoint interiors. The annuli of primary interest come from curves $\beta \in \cc{S}$. More precisely, every such curve $\beta$ determines a conjugacy class of cyclic subgroups of $\pi_1S$ and hence an annular cover (unique up to isomorphism) $A_\beta \to S$ for which $\beta$ lifts to the core curve.

We will often view curve complexes as metric spaces by making every simplex a regular Euclidean simplex with edge lengths $1$. We treat arc and curve complexes and arc complexes similarly. The $1$-skeleton of a curve complex is called the \emph{curve graph}. Analogously, we define the \emph{arc and curve graph} and the \emph{arc graph}.

\subsection{Mapping class group and the Birman exact sequence}
We recall that the mapping class group of $S$ is the group of orientation preserving homeomorphisms (or diffeomorphisms) of $S$, modulo the normal subgroup of those homeomorphisms that are isotopic to the identity,

\[
    \mcg{S} = \Homeo^{+}(S)/\Homeo_0(S).
\]
\noindent Every element of $\mcg{S}$ is thus the isotopy class of a homeomorphism.

Recall that we have fixed a basepoint $z \in S$, and $S^z = S \setminus \set{z}$. We write $\Phi: S^z \to S$ for the inclusion map. We refer to the puncture of $S^z$ that accumulates on $z$ via $\Phi$ as the \emph{$z$-puncture}, and $\Phi$ can be thought of as the map that `fills the $z$-puncture back in'.

Consider the finite index subgroup $\mcg{S^z,z} < \mcg{S^z}$ consisting of isotopy classes of homeomorphisms that fix the $z$-puncture. Any homeomorphism $\varphi: S^z \to S^z$ defining an element of $\mcg{S^z,z}$ uniquely determines a homeomorphism $\varphi': S \to S$ extending over the point $z$ by sending $z$ to itself and by the formula $\varphi' \circ \Phi = \Phi \circ \varphi$ on $S^z$. When the context makes the meaning clear, we abuse notation and use the same symbol $\varphi$ to denote the mapping class in $\mcg{S^z,z}$, a representative homeomorphism of $S^z$, as well as the unique extension to a homeomorphism of $S$.

The extension of a homeomorphism of $(S^z,z)$ over the point $z$ via the map $\Phi$ defines a surjective homomorphism $\Phi_*: \mcg{S^z,z} \to \mcg{S}$, and the \emph{Birman exact sequence} \cite{MR0243519} gives an isomorphism of the kernel of $\Phi_*$ with $\pi_1S$,

\[
\begin{tikzcd}
    1 \ar{r} & \pi_1S \ar{r} & \mcg{S^z,z} \ar{r}{\Phi_*} & \mcg{S} \ar{r} & 1.
\end{tikzcd}
\]

It will be useful to describe explicitly the isomorphism of the kernel of $\Phi_*$ with $\pi_1S$. If $\varphi: S^z \to S^z$ represents an element of the kernel, then the extension $\varphi: S \to S$ over the point $z$ is isotopic to the identity via an isotopy that need not preserve $z$. If $\varphi_t: S \to S$ is the isotopy so that $\varphi_0 = \varphi$ and $\varphi_1 = \id_S$, then defining $\gamma(t) = \varphi_t(z)$ gives a loop $\gamma$ based at $z$. The isomorphism of the kernel with $\pi_1S$ assigns the homotopy class of $\gamma$ to $\varphi \in \mcg{S^z,z}$. Alternatively, we can think of producing a homeomorphism $\varphi: S^z \to S^z$ by `pushing' $z$ around the loop $\gamma^{-1}$ by an isotopy on $S$; we call this the \emph{point push around $\gamma^{-1}$}.

Another perspective is useful in our setting. Fix a lift $\widetilde{z} \in p^{-1}(z)$. Any mapping class representative $\varphi: S^z \to S^z$ has a unique lift $\widetilde{\varphi}: \HH^2 \to \HH^2$ fixing $\widetilde{z}$. The lift $\widetilde{\varphi}$ is a quasi-isometry, and so has a unique extension to a homeomorphism $\partial \HH^2 \to \partial \HH^2$. Any other representative of the isotopy class of $\varphi$ in $\mcg{S^z,z}$ has the same extension to the boundary, since the lift of the isotopy moves all points a bounded hyperbolic distance. Thus, we obtain an action of $\mcg{S^z,z}$ on $\partial \HH^2$.

Next, observe that if $\varphi_0: S^z \to S^z$ represents an element in the kernel of $\Phi_*$, and $\varphi_t: S \to S$ is the isotopy to the identity. This isotopy lifts to an isotopy $\widetilde{\varphi}_t$ from the lift $\widetilde{\varphi}_0$ fixing $\widetilde{z}$ to a lift of the identity. The resulting lift of the identity $\widetilde{\varphi}_1$ is thus a covering transformation, namely the one associated to the homotopy class of $\gamma$ (where $\gamma(t) = \varphi_t(z)$ as defined above). Thus, we have the following proposition.

\begin{proposition}[\cite{MR2851869,MR4632569}]
    The restriction of the action of $\mcg{S^z,z}$ on $\partial \HH^2$ to $\pi_1S$ agrees with the extension of the isometric covering action of $\pi_1S$ on $\HH^2$.
\end{proposition}

Kra's Theorem \cite{MR0611385} describes precisely which elements of $\pi_1S$ represent pseudo-Anosov elements of $\mcg{S^z,z}$. Recall that a loop is \emph{filling} if it cannot be homotoped to be disjoint from any essential simple closed curve. (Thus, it is a property of the homotopy class.)

\begin{theorem}[\cite{MR0611385}] \label{thm:kra}
    An element of $\pi_1S$ represents a pseudo-Anosov element of $\mcg{S^z,z}$ if and only if it is represented by a filling loop.
\end{theorem}

Since being pseudo-Anosov is equivalent to not having any isotopy classes of periodic simple closed curves, the point pushing description of Birman's isomorphism suggests a proof of Theorem \ref{thm:kra}; see \cite{MR2850125}.

\subsection{Fibers and trees}
Let $\ccs{S^z} \subset \cc{S^z}$ denote the subcomplex spanned by curves whose image under $\Phi: S^z \to S$ is an essential curve on $S$. We call the vertices of $\ccs{S^z}$ the surviving curves of $S^z$. Since $\Phi$ maps disjoint curves to disjoint curves, it induces a simplicial, surjective map, which we also denote $\Phi: \ccs{S^z} \to \cc{S}$ by an abuse of notation. Give any simplex $v \subset \cc{S}$, we let $\Phi^{-1}(v)$ denote the preimage of the barycenter of $v$.

For any simplex $\alpha \subset \cc{S}$, let $T_\alpha$ denote the \emph{Bass-Serre tree dual to $p^{-1}(\alpha)$ in $\HH^2$}. More precisely, $T_\alpha$ contains a vertex for each component of $\HH^2 \setminus p^{-1}(\alpha)$, and this vertex and component of $\HH^2 \setminus p^{-1}(\alpha)$ are said to be \emph{dual} to each other. Two vertices $t_1$ and $t_2$ are connected by an edge if and only if the closures of the components dual to $t_1$ and $t_2$ intersect along some component of $p^{-1}(\HH^2)$, and this edge and component of $p^{-1}(\alpha)$ are said to be \emph{dual} to each other.

We have the following useful theorem relating fibers of $\Phi$ and Bass-Serre trees.
\begin{theorem}[\cite{MR2599078}] \label{thm:tree-curve-complex}
    For any simplex $v \subset \cc{S}$, there is a $\pi_1S$-equivariant homeomorphism from the Bass-Serre tree $T$ dual to $v$ to $\Phi^{-1}(v) \subset \ccs{S^z}$. The image of a vertex $t \in T$ under this homeomorphism is the barycenter of a simplex $u_t \subset \ccs{S^z}$ for which $\Phi(u_t) = v$ and $\Phi|_{u_t}$ is injective. Moreover, $t,t' \in T$ are joined by an edge if and only if $u_t \cup u_{t'}$ span a simplex of $\ccs{S^z}$.
\end{theorem}

The proof of Theorem \ref{thm:tree-curve-complex} involves some ideas that will be useful for us, which we briefly describe. Given a simplex $u \subset \cc{S^z}$, we let $K_u$ denote the stabilizer of $u$ in $\pi_1S < \mcg{S^z,z}$ and let $\hh_u \subset \HH^2$ denote the convex hull of the limit set of $K_u$ in $\partial \HH^2$ (if it is non-empty). If $u \subset \ccs{S^z}$, $v = \Phi(u)$, and $\Phi \mid_u$ is injective, then $p: \HH^2 \to S$ maps the interior $\hh_u^\circ \subset \hh_u$ to a component of $S \setminus v$ (where $v$ is realized by its geodesic representative). Up to isotopy, $p(\hh_u^\circ)$ is the $\Phi$-image of the component $U \subset S^z \setminus u$ containing the $z$-puncture. One way to think about this fact is that point pushing around a loop preserves $u$ precisely when the loop is disjoint from $u$, that is, when the loop (intersected with $S^z$) is contained in $U$. When $\Phi \mid_u$ is not injective, the component of $S^z \setminus u$ containing the $z$-puncture is a once-punctured annulus, making $K_u$ an infinite cyclic group. In any case, the stabilizer of $\hh_u$ is exactly $K_u$; see \cite{MR2599078}.

\subsection{Convex cocompactness} \label{sec:ccc}
Farb and Mosher originally defined convex cocompactness in the mapping class group using the action on Teichm\"{u}ller space; see \cite{MR1914566}. For our purposes, it will be more convienient to use the following alternative formulation due to Kent-Leininger and independently Hamenst\"{a}dt.

\begin{theorem}[\cite{MR2465691,MR2349677}] \label{thm:qi-ccc}
    A finitely generated subgroup $G$ of $\mcg{S}$ is convex cocompact if and only if the orbit map $G \to G \cdot u$ is a quasi-isometric embedding into the curve complex $\cc{S}$.
\end{theorem}

We will apply this to the case of subgroups of $\mcg{S^z,z}$. We note that since the inclusion of a finite index subgroup is a quasi-isometry, convex cocompactness is shared amongst groups which differ only by finite index.

\section{Setup}
Let $H < \mcg{S}$ be a free abelian subgroup of rank $n$, and let $\Gamma_H < \mcg{S^z,z}$ be the full preimage of $H$ under $\Phi_*: \mcg{S^z,z} \to \mcg{S}$. As the low rank cases have been answered by \cite{MR2599078,MR3314946,MR4632569}, we restrict our attention to $H$ of rank $n \geq 2$. Take a finitely generated and purely pseudo-Anosov subgroup $G < \Gamma_H$. Our task is to show that $G$ must necessarily be convex cocompact. To begin, we reduce the problem to only consider subgroups $G < \Gamma_H$ for which the restricted homomorphism $\Phi_*|_G: G \to H$ is surjective onto $H$.

If $G < \Gamma_H$ is a subgroup such that $\Phi_*(G) \neq H$, then $H' = \Phi_*(G)$ must be some proper subgroup of $H$. Since subgroups of free abelian groups are free abelian, $H'$ is also a free abelian subgroup of $\mcg{S}$. Further, $G < \Gamma_{H'}$ and $G$ surjects onto $H'$ via $\Phi_*$. Thus, it suffices to prove Theorem \ref{thm:main2} for subgroups of $\Gamma_H$ that surject onto $H$ via $\Phi_*$.

Now assuming $\Phi_*|_G$ is surjective, we know that $\Phi_*|_G$ cannot also be injective. If it were, then $G$ would be isomorphic to a free abelian group of rank $n \geq 2$ and such a group cannot be purely pseudo-Anosov, or even irreducible. Thus, $\Phi_*|_G$ must have some nontrivial kernel, and we call this nontrivial normal subgroup $G_0 = \ker{\Phi_*|_G} = G \cap \pi_1S$. Let $\phi: G/G_0 \to H$ denote the isomorphism from the quotient to $H$.

\subsection{Realizing $H$ by homeomorphisms} \label{sec:realizable}
We will see that $H$ has a finite index normal subgroup that can be realized by homeomorphisms. Keeping Lemma \ref{lem:finiteindex} in mind, this is enough for our argument. We use to great effect the reduction system machinery of Birman-Lubotsky-McCarthy and Ivanov; see \cite[Chapter 7]{MR1195787} for more details.

Since $H$ is free abelian of rank $n \geq 2$, it must be reducible as a subgroup of $\mcg{S}$. This is because if $H$ were irreducible, it would have to contain a pseudo-Anosov mapping class whose centralizer in $\mcg{S}$ must contain $H$. However, this contradicts the fact that the centralizer of a pseudo-Anosov is virtually cyclic \cite{mccarthy1982normcent}. Thus, $H$ must be reducible and have a non-empty canonical reduction system. Let $\alpha$ denote this canonical reduction system. We assume throughout that $\alpha$ is realized as a geodesic multicurve in $S$ with respect to our fixed hyperbolic metric. Write $\alpha_1, \cdots, \alpha_m$ to denote the components of $\alpha$, and let $A_j \to S$ be the annular cover with core curve a lift of $\alpha_j$. We reserve the symbol $j$ to index objects associated to the reducing curve $\alpha_j$.

A \emph{complementary subsurface} to the canonical reduction system $\alpha$ is defined as the path metric completion $Y$ of a component $Y^\circ \subset S \setminus \alpha$. Such a complementary subsurface $Y$ is a hyperbolic surface with geodesic boundary, and the inclusion $Y^\circ \to S \setminus \alpha$ extends to an immersion $Y \to S$ which is injective on the interior and at most $2$-to-$1$ on $\partial Y$. By an abuse of notation, we often write $Y \subset S$ or refer to the map $Y \to S$ as the inclusion. Write $Y_1, \ldots, Y_l$ to denote the subsurfaces complementary to $\alpha$. We reserve the symbol $k$ to index objects associated to the complementary subsurface $Y_k$. Given a complementary surface $Y_k$, we continue to abuse notation and identify each component $\delta \subset \partial Y_k$ with the reducing curve $\alpha_j \subset S$ that is the image of $\delta$ under the immersion $Y_k \to S$. With this notation, we may write $\alpha_j \subset \partial Y_k$. In cases where $Y_k \to S$ is injective on the boundary, distinct components of $\partial Y_k$ map to distinct $\alpha_j$. However, in cases where $Y_k \to S$ is $2$-to-$1$ on $\alpha_j$, there may be some pairs of distinct components $\delta, \delta' \subset \partial Y_k$ which are identified with the same $\alpha_j$. In any circumstance where this notation might cause confusion, we explicitly clarify the situation.

Let $H' < H$ be the finite index normal subgroup of $H$ obtained by taking the intersection of $H$ with the kernel of the natural homomorphism $\mcg{S} \to \Aut(H_1(S,\ZZ/3\ZZ))$. This subgroup $H'$ has the same canonical reduction system as $H$ plus a number of other useful properties, as shown by Ivanov \cite[Chapter 1]{MR1195787}. Every element of $H'$ has a \emph{pure} representative which fixes the reducing system $\alpha$ pointwise and preserves each complementary subsurface $Y_k$. Thus, we have well defined homomorphisms $\rho_k: H' \to \mcg{Y_k}$ given by taking the restriction of a pure representative $h \mapsto h|_{Y_k}$. The additional property of being pure helps us to realize $H'$ by explicit homeomorphisms. However, we require $H'$ to have another extra property for our purposes, so we will take a further finite index subgroup to achieve this property before proceeding.

Consider for each subsurface $Y_k$ the image subgroups $\rho_k(H') < \mcg{Y_k}$. Ivanov proves that $\rho_k(H')$ must be either trivial or irreducible \cite[Theorem 7.18]{MR1195787} and torsion-free \cite[Lemma 1.6]{MR1195787}. Further, $\rho_k(H')$ must be abelian, since $H'$ is abelian. The only irreducible, torsion-free, abelian subgroups of mapping class groups are infinite cyclic groups generated by a pseudo-Anosov mapping class. For each $Y_k$ with nontrivial $\rho_k(H')$, let $\psi_k$ be a pseudo-Anosov mapping class generating $\rho_k(H')$, and call $Y_k$ a \emph{pseudo-Anosov component}. For each $Y_k$ with trivial $\rho_k(H')$, let $\psi_k$ be the identity mapping class, and call $Y_k$ an \emph{identity component}.

Given a pseudo-Anosov component $Y_k$, we realize the pseudo-Anosov mapping class $\psi_k$ by a representative homeomorphism $\psi_k$ that preserves a pair of transverse measured geodesic laminations $(\Lambda^s, \mu^s)$ and $(\Lambda^u, \mu^u)$ called the stable and unstable laminations. Taking either lamination to be $\Lambda$, any component of $Y_k \setminus \Lambda$ containing a component of $\partial Y_k$ is a semi-open annulus whose path metric completion is a \emph{crown}; see \cite[Section 4]{MR964685}. Given a component $\delta \subset \partial Y_k$, let $C_\delta$ denote the crown obtained from the component of $Y_k \setminus \Lambda$ containing $\delta$. The boundary $\partial C_\delta$ consists of $\delta$ and a finite number $d$ of bi-infinite geodesics $\lambda_1, \ldots, \lambda_d$, labeled so that the index of $\lambda_i$ increases (modulo $d$) in the direction of the boundary orientation of $\delta$. Let $P_i \subset C_\delta$ be the geodesic ray beginning on and orthogonal to $\delta$ which extends out infinitely along the `spike' of the crown between $\lambda_i$ and $\lambda_{i+1}$ (modulo $d$). We call these $P_i$ \emph{prongs}. By modifying $\psi_k$ via an isotopy in the complement of $\Lambda$ if necessary, we can arrange to have $\psi_k$ preserve the set of prongs $\set{P_i}$. Note that $\psi_k$ may rotate the prongs. More precisely, if $x_i = P_i \cap \delta$, then there exists an integer $c, ~0 \leq c < d$ so that $\psi_k(x_i) = x_{i+c}$ (modulo $d$) for all $i$. We call this value $c = c(\psi_k,\delta) \in \ZZ/d\ZZ$ the \emph{crown shift} of the pseudo-Anosov homeomorphism $\psi_k$ along $\delta \subset \partial Y_k$. For our argument, we want to modify the $\psi_k$ to have a crown shift of $0$ along every boundary component.

\begin{remark} \label{rmk:fracdehntwist}
    This crown shift $c(\psi_k,\delta)$ is very much related to the \emph{fractional Dehn twist coefficient} of a pseudo-Anosov mapping class $f \in \mcg{Y_k, \partial Y_k}$ relative to $\delta \subset \partial Y_k$. In particular, if $f$ is a rel boundary mapping class with representatives free isotopic to $\psi_k$, then its fractional Dehn twist coefficient relative to $\delta$ must be congruent to $\frac{c}{d}$ modulo $1$; see \cite[Section 3.2]{MR2318562}.
\end{remark}

For each pseudo-Anosov component $Y_k$, consider each component $\delta \subset \partial Y_k$ for which the image of $\delta$ under the inclusion $Y_k \subset S$ is a reducing curve. For each of these $\delta$, observe the crown and note down the number of prongs $d(\psi_k,\delta)$. Let $b$ be the lowest common multiple of all such $d$. Fixing a minimal generating set $\set{f_1, \ldots, f_n}$ for $H'$, we consider the finite index subgroup $H''$ generated by $\set{f_1^b, \ldots, f_n^b}$. Since $\rho_k(H')$ is generated by a mapping class isotopic to $\psi_k$, we know that $\rho_k(H'')$ is generated by a mapping class isotopic to $\psi_k^b$. In addition, if $c$ is the crown shift of $\psi_k$ along $\delta$, then the crown shift of $\psi_k^b$ along $\delta$ must be $bc$ modulo $d$. Since $b$ is a multiple of $d$, we have $c(\psi_k^b,\delta) = 0$.

We now extend $\psi_k^b$ to a homeomorphism on $S$. We do this by first choosing a rel boundary mapping class $\overline{\psi}_k \in \mcg{Y_k, \partial Y_k}$ with representatives free isotopic to $\psi_k^b$. As mentioned in Remark \ref{rmk:fracdehntwist}, the fractional Dehn twist coefficient of any such $\overline{\psi}_k$ relative to each $\delta \subset \partial Y_k$ yields a number congruent to $\frac{c}{d}$ modulo $1$. Since $c(\psi_k^b,\delta) = 0$, the fractional Dehn twist coefficients here must be integers. We make the choice of $\overline{\psi}_k$ which has fractional Dehn twist coefficient $0$ relative to each $\delta$. Any other choice can be obtained by composing with powers of Dehn twists about the boundary curves of $Y_k$. Next, we realize our chosen mapping class $\overline{\psi}_k$ by a representative homeomorphism free isotopic to $\psi_k^b$ via an isotopy supported on $Y_k \setminus \Lambda$. We can again arrange that $\overline{\psi}_k$ preserves the prongs of each crown. Now we have a homeomorphism $\overline{\psi}_k: Y_k \to Y_k$ that fixes $\partial Y_k$ pointwise and preserves the prongs of each crown with crown shift $0$. Finally, we extend $\overline{\psi}_k: Y_k \to Y_k$ to a homeomorphism on $\widehat{\psi}_k: S \to S$ by gluing with the identity map on $S \setminus Y_k$. Continuing to abusing notation, we write $\widehat{\psi}_k$ to refer to both the homeomorphism and its mapping class in $\mcg{S}$. Note that for each identity component, $\widehat{\psi}_k$ is just the identity map on $S$.

Let $\tau_j$ denote the right-handed Dehn twist about $\alpha_j$, and let $\widehat{H} < \mcg{S}$ be the subgroup
\[
\widehat{H} = \gen{\tau_1, \ldots, \tau_m, \widehat{\psi}_1, \ldots, \widehat{\psi}_l}.
\]
Since the nontrivial generators are supported on disjoint subsurfaces, $\widehat{H}$ is free abelian of rank $r \leq m+l$ with the deficiency being exactly the number of trivial generators (or identity components). By construction, the pure representative of each mapping class in $H''$ is some product of powers of the generators of $\widehat{H}$, so $H'' < \widehat{H}$. To reiterate, each $h \in H''$ can be decomposed as a product of generators
\begin{equation} \label{eq:decomposition}
    h = \tau_1^{q_1} \circ \cdots \circ \tau_m^{q_m} \circ \widehat{\psi}_1^{q_{m+1}} \circ \cdots \circ \widehat{\psi}_l^{q_{m+l}},
\end{equation}
and the decomposition is unique after omitting any trivial generators.

\begin{remark} \label{rmk:k'notation}
To help with indexing objects associated to $Y_k$ but offset by $m$, we introduce the symbol $k' := m+k$. As an example, in the decomposition above, the number $q_j$ is the power of $\tau_j$ with $j \in \oneto{m}$, and the number $q_{k'}$ is the power of $\widehat{\psi}_k$ with $k \in \oneto{l}$ and $k' = m+k$.    
\end{remark}

Each nontrivial $\widehat{\psi}_k$ is realized by a homeomorphism supported on $Y_k$. We can adjust these homeomorphisms via isotopy to be supported on marginally smaller copies of $Y_k$ that are disjoint from small annular neighborhoods around each $\alpha_j$. We can also realize each $\tau_j$ by Dehn twists supported on these small annular neighborhoods. Since the nontrivial generators of $\widehat{H}$ can be realized by homeomorphisms which are supported on disjoint surfaces and thus commute, we obtain an injective homomorphism $\widehat{H} \to \Homeo^+(S)$. Thus, the subgroup $H'' < \widehat{H}$ can also be realized by homeomorphisms.

\begin{remark}
    By replacing $\phi_k^b$ with $\phi_k$ in the construction above, one can just as readily show that $H'$ can be realized by homeomorphisms. The key difference is that restrictions of $h \in H'$ to subsurfaces may have nonzero crown shift, while restrictions of $h \in H''$ must have zero crown shift. This additional property (and the group $\widehat{H}$) will be necessary in Section \ref{sec:Haction}.
\end{remark}

\begin{remark}
    The question of which subgroups $H < \mcg{S}$ are realizable in $\Homeo^+(S)$ -- sometimes called the generalized Nielsen realization problem or the section problem -- has been of interest to many in the literature. Farb mentions this problem in \cite[Chapter 2, Section 6.3]{MR2264130} and remarks that the work of Birman-Lubotsky-McCarthy can be used to prove free abelian $H$ are always realizable. Mann and Tshishiku give more details in \cite[Section 4.2]{MR3967366}, and describe many ideas relevant to our own approach here. Thus, it was already well-known that general free abelian $H$ are realizable. Nevertheless, we still find our particular approach here using the specific case of \emph{pure} free abelian subgroups necessary, because the objects and properties therein remain important for the remainder of the work.
\end{remark}

Lemma \ref{lem:finiteindex} tells us that can always replace $H$ with the finite index subgroup $H''$. In practice, we may as well have assumed that $H = H''$ in the first place. We adopt this assumption for the remainder of the work. We have just shown that $H$ can be realized by homeomorphisms. Fix a minimal generating set $\set{h_1, \ldots, h_n}$ for $H$, and realize these generators by homeomorphisms as above, so that they generate a copy of $H$ in $\Homeo^+(S)$. Since the homeomorphisms of $H$ have common fixed points (such as on the boundary of small annular neighborhoods around each $\alpha_j$), we may assume that $H$ fixes our basepoint $z \in S$ (up to conjugation by a homeomorphism isotopic to the identity). Now consider the action of $H$ on $\pi_1(S,z)$ which induces a homomorphism $H \to \Aut(\pi_1(S,z))$. The short exact sequence
\[
    \begin{tikzcd}
        1 \ar{r} & \pi_1S \ar{r} & \Gamma_H \ar{r} & H \ar{r} & 1
    \end{tikzcd}
\]
splits with respect to this homomorphism, and $\Gamma_H \cong \pi_1S \rtimes \ZZ^n$.

\subsection{$H$ acting on subsurfaces and annuli} \label{sec:Haction}
Each mapping class $h \in \widehat{H}$ acts on the annular covers $A_j$ (via the lift $\widetilde{h}_j$) and on the complementary subsurfaces $Y_k$ (via the restriction $h \mapsto h|_{Y_k}$). These actions further induce actions on the corresponding arc and curve complexes $\ac{A_j}$ and $\acc{Y_k}$. We say that a mapping class $h \in \widehat{H}$ acts \emph{coarsely as the identity} on an arc graph $\ac{A_j}$ if there exists some uniform bound $b > 0$ such that $d(\beta,\widetilde{h}_j(\beta)) \leq b$ for any vertex $\beta \in \ac{A_j}$. Equivalently, the sup-distance to the identity is uniformly bounded. The following lemma states that the action of $\widehat{H}$ on these subsurfaces and annuli are `coarsely diagonal' with respect to decomposition \ref{eq:decomposition}.

\begin{lemma} \label{lem:diagonalaction}
    Each $\tau_j$ acts
    \begin{enumerate}
        \item loxodromically on $\ac{A_j}$;
        \item coarsely as the identity on each $\ac{A_i}$ for $i \neq j$ (with uniform bound $2$);
        \item as the identity on each $\acc{Y_k}$.
    \end{enumerate}
    Each nontrivial $\widehat{\psi}_k$ acts
    \begin{enumerate}
        \item loxodromically on the $\acc{Y_k}$;
        \item as the identity on each $\acc{Y_i}$ for $i \neq k$;
        \item coarsely as the identity on each $\ac{A_j}$ (with uniform bound $2$).
    \end{enumerate}
\end{lemma}
\begin{proof}
    We begin with the 3 statements for $\tau_j$. For statement (1), the lift of $\tau_j$ to $A_j$ must twist about the core curve of $A_j$, and such homeomorphisms act loxodromically on $\ac{A_j}$; see \cite{MR1714338, MR1791145}. For statement (2), we first find a vertex $\gamma \in \ac{A_i}$ fixed by the lift of $\tau_j$. Choose a geodesic curve on $S \setminus \alpha_j$ that intersects $\alpha_i$, and then choose a component of the preimage of this curve in the annular cover $A_i \to S$. Take the isotopy class of the resulting arc to be $\gamma$. Since the original curve on $S$ was disjoint from $\alpha_j$, the lift of $\tau_j$ fixes $\gamma$. We now appeal to the following claim.
    
    \begin{claim} \label{clm:coarseidentity}
        Any lift $\widetilde{h}$ that fixes a vertex $\gamma \in \ac{A_j}$ acts coarsely as the identity on $\ac{A_j}$ (with uniform bound $2$).
    \end{claim}
    \begin{proof}[Proof of claim]
        Suppose $\widetilde{h}$ fixes $\gamma \in \ac{A_j}$, and take any other vertex $\beta \in \ac{A_j}$. Let $i(\cdot,\cdot)$ denote the signed intersection number of geodesic representatives. Since $\widetilde{h}(\gamma) = \gamma$, we have that $i(\beta, \gamma) = i(\widetilde{h}(\beta), \widetilde{h}(\gamma)) = i(\widetilde{h}(\beta), \gamma)$. Any two essential geodesic arcs on $A_j$ which have the same signed intersection number with a common arc can intersect each other at most once. (One can see this via a linking argument in the universal cover.) Thus, $|i(\beta, \widetilde{h}(\beta))| \leq 1$. Since distances in $\ac{A_j}$ are given by geometric intersection number plus $1$, this means $d(\beta, \widetilde{h}(\beta)) \leq 2$. We have shown a uniform bound on the distance that $\widetilde{h}$ can send an arbitrary vertex of $\ac{A_j}$.
    \end{proof}
    
    Statement (3) splits into two cases. If $\alpha_j \not \subset \partial Y_k$, then the support of $\tau_j$ is disjoint from $Y_k$ and thus fixes each arc and curve on $Y_k$. If $\alpha_j \subset \partial Y_k$, then $\tau_j$ will affect arcs on $Y_k$ with an endpoint on $\alpha_j$. Nevertheless, such arcs only differ from their $\tau_j$-image by a free isotopy which undoes the boundary twist, so $\tau_j$ still fixes each isotopy class of arcs and curves on $Y_k$.

    We now continue to the 3 statements for nontrivial $\widehat{\psi}_k$. For statement (1), the restriction of $\widehat{\psi}_k$ to $Y_k$ is the pseudo-Anosov mapping class $\psi_k$, and pseudo-Anosovs act loxodromically on $\acc{Y_k}$; see \cite{MR1714338, MR1791145}. For statement (2), the restriction of $\widehat{\psi}_k$ to any $Y_i$ with $i \neq k$ is the identity by construction. For statement (3), we again need only find a vertex $\gamma \in \ac{A_j}$ which is fixed by the lift of $\widehat{\psi}_k$ and then apply Claim \ref{clm:coarseidentity}. There are two possible cases.

    If $\alpha_j \not \subset \partial Y_k$, choose a geodesic curve on $S \setminus Y_k$ that intersects $\alpha_j$, and then choose a component of the preimage of this curve in the annular cover $A_j \to S$. Take the isotopy class of the resulting arc to be $\gamma$. Since the original curve on $S$ was disjoint from $Y_k$, the lift of $\widehat{\psi}_k$ fixes $\gamma$.

    Now suppose $\alpha_j \subset \partial Y_k$. Note that $\alpha_j$ is contained in the boundary of at most two complementary subsurfaces, with one of them being $Y_k$. If the immersion $Y_k \to S$ is injective on $\alpha_j$, then there is another subsurface $Y_i$ such that $\alpha_j \subset \partial Y_i$ and $i \neq k$. If the immersion $Y_k \to S$ is $2$-to-$1$ on $Y_k$, then we will still write $\alpha_j \subset \partial Y_i$ but now $i = k$. We first proceed with the argument in the case $i \neq k$.
    
    Just as $\alpha_j$ lies between the two subsurfaces $Y_i$ and $Y_k$ in $S$, the core curve of $A_j$ lies between two subsurfaces (complementary to all lifts of each component of $\alpha$), one of which is a lift of $Y_i$ and the other a lift of $Y_k$. We construct an essential arc $\gamma \subset A_j$ by joining two rays that start from the core curve and travel out to distinct boundary components of $\overline{A}_j$.

    We first describe a ray $\gamma_1$ which travels into the interior of the lift of $Y_i$ adjacent to the core curve. Choose a geodesic ray on $S \setminus Y_k$ that starts on $\alpha_j$ and continues on forever (for infinite distance). For example, the ray might limit to some closed curve on $S \setminus Y_k$. Choose a preimage of this ray in the annular cover $A_j \to S$, and take the resulting arc to be $\gamma_1$. Since the original ray on $S$ was chosen disjoint from $Y_k$, any lift of $\widehat{\psi}_k$ fixes $\gamma_1$.

    Next, we describe a ray $\gamma_2$ which travels into the interior of the lift of $Y_k$ adjacent to the core curve. The lamination $\Lambda \subset Y_k$ lifts to a lamination $\widetilde{\Lambda}$ in the lift of $Y_k$ adjacent to the core curve, and similarly the crown $C$ along $\alpha_j$ lifts to a crown $\widetilde{C}$ along the core curve. Choose any prong of $\widetilde{C}$ to be $\gamma_2$. Since $\overline{\psi}_k$ was constructed to have $0$ crown shift and chosen to have $0$ fractional Dehn twist coefficient, any lift of $\widehat{\psi}_k$ must preserve $\gamma_2$, at least up to isotopy in the complement of $\widetilde{\Lambda}$.
    
    Finally, we join $\gamma_1$ and $\gamma_2$ perhaps by some subsegment $\gamma_3$ of the core curve to obtain an essential arc $\gamma$ whose isotopy class is preserved by $\widehat{\psi}_k$.

    Now in the case where $i = k$, we again join two rays to produce our desired $\gamma$, but now both rays will be of type $\gamma_2$ above. More precisely, if $\delta, \delta' \subset \partial Y_k$ are distinct components which both map to $\alpha_j$ under the immersion $Y_k \to S$, then the crowns $C_\delta, C_{\delta'}$ lift to crowns $\widetilde{C_\delta}, \widetilde{C_{\delta'}}$. We choose any prong of $\widetilde{C_\delta}$ to be $\gamma_1$, and any prong of $\widetilde{C_{\delta'}}$ to be $\gamma_2$. The rest of the argument remains the same.
\end{proof}

Note that a complementary component $Y_k$ is a pseudo-Anosov component precisely if $H$ contains an element $h$ which acts loxodromically on $\acc{Y_k}$, and an identity component otherwise. Similarly, we will say that a component $\alpha_j \subset \alpha$ is \emph{twist} if $H$ contains any element $h$ which acts loxodromically on $\ac{A_j}$, and \emph{non-twist} otherwise. For any $h \in H$, let
\[
    V(h) = \begin{bmatrix} q_1 \\ \vdots \\ q_{m+l} \end{bmatrix} \in \ZZ^{m+l},
\]
where $q_j$ for $j \in \oneto{m}$ and $q_{k'}$ for $k \in \oneto{l}$ are the powers with respect to decomposition \ref{eq:decomposition}. Observe $q_j = 0$ if $\alpha_j$ is non-twist. We choose the convention $q_{k'} = 0$ if $Y_k$ is an identity component. With this choice, the map $V: H \to \ZZ^{m+l}$ is a well-defined, injective homomorphism. Using Lemma \ref{lem:diagonalaction}, we note that an element $h \in H$ acts loxodromically on $\ac{A_j}$ precisely when the entry $q_j$ of $V(h)$ is nonzero. Similarly, $h \in H$ acts loxodromically on $\acc{Y_k}$ precisely when the entry $q_{k'}$ of $V(h)$ is nonzero.

Let $\widehat{\rho}_j: H \to \gen{\tau_j}$ be the composition of the inclusion $H \to \widehat{H}$ and the projection $\widehat{H} \to \gen{\tau_j}$ defined by our chosen basis, and let $H_j = \ker \widehat{\rho}_j$. The kernel $H_j$ contains precisely the group elements $h$ which act coarsely as the identity on $\ac{A_j}$. When $\alpha_j$ is twist, the quotient $H/H_j$ is isomorphic to some infinite cyclic subgroup of $\gen{\tau_j}$. Let $\overline{G}_j = \phi^{-1}(H_j)$ be the preimage of $H_j$ under the isomorphism $\phi: G/G_0 \to H$. When $\alpha_j$ is twist, we have $(G/G_0)/ \overline{G}_j \cong \ZZ$.

Similarly, let $\widehat{\rho}_{k'}: H \to \gen{\widehat{\psi}_k}$ be the composition of the inclusion $H \to \widehat{H}$ and the projection $\widehat{H} \to \gen{\widehat{\psi}_k}$ defined by our chosen basis, and let $H_{k'} = \ker \widehat{\rho}_{k'}$. The kernel $H_{k'}$ contains precisely the group elements $h$ which act coarsely as the identity on $\acc{Y_k}$. When $Y_k$ is a pseudo-Anosov component, the quotient $H/H_{k'}$ is isomorphic to some infinite cyclic subgroup of $\gen{\widehat{\psi}_k}$. Let $\overline{G}_{k'} = \phi^{-1}(H_{k'})$ be the preimage of $H_{k'}$ under the isomorphism $\phi: G/G_0 \to H$. When $Y_k$ is pseudo-Anosov, we have $(G/G_0)/ \overline{G}_{k'} \cong \ZZ$.

\subsection{$\Gamma_H$ acting on $\HH^2$ and T} \label{sec:gammaaction}
We now move to analyzing the group $\Gamma_H$ containing $G$. Recall that we have identified $\HH^2$ with the universal cover $p: \HH^2 \to S$. The covering space action of $\pi_1S$ on $\HH^2$ extends to an action of $\Gamma_H \cong \pi_1S \rtimes \ZZ^n$, as we now explain.

Recall that we have fixed a generating set $\set{h_1, \ldots, h_n}$ for $H$, realized by homeomorphisms fixing the basepoint $z$. Now given any mapping class $\varphi \in \Gamma_H < \mcg{S^z,z}$, we abuse notation and write $\varphi: S^z \to S^z$ for a representative homeomorphism and also write $\varphi: S \to S$ for its extension obtained by filling the $z$-puncture back in. Now $\varphi: S \to S$ is a representative of the mapping class $\Phi_*(\varphi) \in H$. Thus $\varphi: S \to S$ is isotopic to some product of powers generators $h_1^{a_1} \circ \cdots \circ h_n^{a_n}$. The lift $\widetilde{\varphi}: \HH^2 \to \HH^2$ fixing $\widetilde{z}$ is isotopic to a lift of $h_1^{a_1} \circ \cdots \circ h_n^{a_n}$ (not necessarily fixing $\widetilde{z}$), and so these two lifts have the same extension to $\partial\HH^2$. Given lifts $\widetilde{h}_i: \HH^2 \to \HH^2$ for each $h_i$, any lift of $h_1^{a_1} \circ \cdots \circ h_n^{a_n}$ can be obtained by composing $\widetilde{h}_1^{a_1} \circ \cdots \circ \widetilde{h}_n^{a_n}$ with an element of $\pi_1S$. Conversely, any such composition is a lift of $h_1^{a_1} \circ \cdots \circ h_n^{a_n}$. Thus, the action of $\Gamma_H$ on $\partial\HH^2$ factors through an isomorphism with the group $\gen{\widetilde{h}_1, \ldots, \widetilde{h}_n, \pi_1S}$ acting on $\partial\HH^2$. This isomorphism $\Gamma_H \cong \gen{\widetilde{h}_1, \ldots, \widetilde{h}_n, \pi_1S}$ then defines an action on $\HH^2$ extending the covering action of $\pi_1S$. Alternatively, each given lift $\widetilde{h}_i$ is equivariantly isotopic to the lift $\widetilde{\varphi}_i$ of some $\varphi_i \in \Gamma_H$ with $\Phi_*(\varphi_i) = h_i$. Then $\Gamma_H = \pi_1S \rtimes \gen{\varphi_1, \ldots, \varphi_n}$ acts on $\HH^2$ so that $\pi_1S$ acts by covering transformations and $\varphi_i$ acts by $\widetilde{h}_i$. Note that while the $\pi_1S$ part of this action is by isometries, the full $\Gamma_H$-action on $\HH^2$ \underline{is not} by isometries.

By the Collar Lemma \cite{MR0379833}, we can assume that our fixed hyperbolic metric on $S$ is chosen so that the lengths of the $\alpha_j$ are short enough to guarantee that any two components of $p^{-1}(\alpha)$ are distance at least $2$ apart in $\HH^2$. Let $T = T_\alpha$ be the Bass-Serre tree dual to $p^{-1}(\alpha)$.

The covering space action of $\pi_1S$ on $\HH^2$ preserves $p^{-1}(\alpha)$, and so $\pi_1S$ also acts on $T$. Further, since each $h_i$ preserves $\alpha$, each $\widetilde{h}_i$ preserves $p^{-1}(\alpha)$, so $\Gamma_H \cong \gen{\widetilde{h}_1,\ldots,\widetilde{h}_n,\pi_1S}$ also acts on the Bass-Serre tree $T$. Since each $h_i$ fixes each curve $\alpha_j$ and each subsurface $Y_k$, a pair of vertices/edges of $T$ are in the same $\Gamma_H$-orbit if and only if they are in the same $\pi_1S$-orbit. Unlike the action of $\Gamma_H$ on $\HH^2$, the action of $\Gamma_H$ on $T$ \underline{is} by isometries.

For each edge $e$ of $T$, we write $\widetilde{\alpha}_e$ to denote the component of $p^{-1}(\alpha)$ that is dual to $e$. We let $K_e = \Stab_{\pi_1S}(\widetilde{\alpha}_e)$ and define $A_e$ to be the annulus $\HH^2/K_e$. We can identify $A_e$ with exactly one of the annular covers $A_1, \ldots, A_m$. There exists a unique index $j(e) \in \oneto{m}$ so that $p(\widetilde{\alpha}_e) = \alpha_{j(e)}$. When convenient, we will also write $\alpha_e = \alpha_{j(e)}$. When $e$ and $e'$ are in the same $\Gamma$-orbit, $A_e$ and $A_{e'}$ are equivalent annular covers of $S$ with core curve $\alpha_e = \alpha_{e'}$. Hence, we can isometrically identify all these annuli: $A_e = A_{j(e)} = A_{j(e')} = A_{e'}$.

For each vertex $t$ of $T$, we write $\widetilde{Y}_t^\circ$ to denote the component of $\HH^2 \setminus p^{-1}(\alpha)$ dual to $t$, and we use $\widetilde{Y}_t$ for its closure. We let $K_t = \Stab_{\pi_1S}(\widetilde{Y}_t)$ and define $Y_t$ to be $\widetilde{Y}_t/K_t$. We can identify each $Y_t$ with exactly one of the complementary subsurfaces $Y_1, \ldots, Y_l$ as follows. For each vertex $t \in T$, let $\Upsilon_t := \HH^2/K_t$. The surface $Y_t$ is then the convex core of $\Upsilon_t$ and there is a unique $k(t) \in \oneto{l}$ so that the covering map $\Upsilon_t \to S$ maps the interior of $Y_t$ isometrically onto $Y_{k(t)}^\circ \in \set{Y_1^\circ, \ldots, Y_l^\circ}$. If $t$ and $t'$ are in the same $\Gamma$-orbit, then $\Upsilon_t$ and $\Upsilon_{t'}$ are equivalent covers of $S$ with different choices of base point. Hence, there is an isomorphism of covering spaces $\Upsilon_t \to \Upsilon_{t'}$ that sends $Y_t$ isometrically to $Y_{t'}$. In particular, $Y_{k(t)} = Y_{k(t')}$, and we use this to identify $Y_t = Y_{k(t)} = Y_{k(t')} = Y_{t'}$. 

We say that an edge $e$ of $T$ is \emph{$\alpha_j$-dual} if $\alpha_e = \alpha_j$. We then also say that $e$ is \emph{twist} if $\alpha_j$ is twist, and \emph{non-twist} if $\alpha_j$ is non-twist. Similarly, we say that a vertex $t$ of $T$ is \emph{$Y_k$-dual} if $Y_t = Y_k$. We then also say that $t$ is a \emph{pseudo-Anosov vertex} if $Y_k$ is a pseudo-Anosov component, and an \emph{identity vertex} if $Y_k$ is an identity component. Since each $G$-orbit is contained in some $\Gamma_H$-orbit, all edges in the same $G$-orbit share the property of being $\alpha_j$-dual or twist, and all vertices in the same $G$-orbit share the property of being $Y_k$-dual or pseudo-Anosov.

We choose a $\pi_1S$-equivariant homeomorphism $T \to \Phi^{-1}(\alpha)$ as in Theorem \ref{thm:tree-curve-complex}, which allows us to identify vertices and edges of $T$ with simplices of $\cc{S^z}$ in $\Phi^{-1}(\alpha)$. If an edge $e$ and vertex $t$ are identified with simplices $u_e$ and $u_t$, then $K_e = K_{u_e}$ and $K_t = K_{u_t}$ are indeed special cases of simplex stabilizers. Moreover, $\widetilde{Y}_t = \hh_{u_t}$ and $\widetilde{\alpha}_e = \hh_{u_e}$. Using this, and the fact that the $\Gamma$-orbits and $\pi_1S$-orbits of vertices and edges of $T$ are the same, it follows that the $\pi_1S$-equivariant map $T \to \Phi^{-1}(\alpha) \subset \cc{S^z}$ is also $\Gamma_H$-equivariant.

After adjusting the homeomorphisms $h_i$ if necessary (via isotopy on small annular neighborhoods around components of $\alpha$), we can also choose a $\Gamma_H$-equivariant map $\HH^2 \to T$ sending $\widetilde{\alpha}_e$ to the midpoint of $e$ and $\widetilde{Y}_t$ to the $\frac{1}{2}$-neighborhood of $t$. There are many such choices, but we make a choice that will be convenient for future application: we choose the map to be $K$-Lipschitz, for some $K>0$. This is possible because of our assumption that the minimal distance between pairs of components of $p^{-1}(\alpha)$ is at least $2$.

\subsection{Invariant subtrees}
We now define invariant subtrees of $T$ associated to simplex stabilizers and our purely pseudo-Anosov subgroup $G < \Gamma_H$. These subtrees will allow us to translate distances in $\cc{S^z}$ to distances in $G$. We begin with the subtrees associated to simplex stabilizers. The lemmas in this section originate from \cite[Section 3.3]{MR4632569}.

For each simplex $u \subset \cc{S^z}$, recall the stabilizer $K_u < \pi_1S$ and its convex hull $\hh_u \subset \HH^2$. The group $K_u$ acts by isometries on $T$. If this action does not have a global fixed point, we let $T^u$ be the minimal $K_u$-invariant subtree of $T$. In this case, $T^u$ is the union of the axes of loxodromic elements; see \cite{MR1886668}. If the $K_u$-action does have a global fixed point in $T$, we instead define $T^u$ to be the maximal fixed subtree.

The following is \cite[Lemma 3.1]{MR4632569}, which shows that much of the structure of $T^u$ can be determined by examining the component of $S^z \setminus u$ that contains the $z$-puncture.
\begin{lemma} \label{lem:multilemma}
    Let $u \subset \cc{S^z}$ be a multicurve and $U$ be the component of $S^z \setminus u$ that contains the $z$-puncture.
    \begin{enumerate}
        \item The action of $K_u$ on $T$ has a global fixed point if and only if $\alpha$ can be isotoped to be disjoint from $\Phi(U)$ in $S$.
        \item When $K_u$ has a global fixed point, $T^u$ is either a single vertex $t \in T$ or a single edge $e \subset T$. Moreover, $T^u$ is an edge $e$ if and only if $U$ is a once-punctured annulus and each component of $\Phi(\partial U)$ is isotopic to the curve $\alpha_e$ of $\alpha$.
        \item If $u$ contains a non-surviving curve, then $T^u$ is a single vertex.
        \item When $u$ consists only of surviving curves and $T^u$ is not an edge, then $t \in T^u$ if and only if $\hh_u \cap \widetilde{Y}_t^\circ \neq \emptyset$.
    \end{enumerate}
\end{lemma}

The invariant subtrees associated to nested simplices must intersect non-trivially, \cite[Lemma 3.2]{MR4632569}, which allows us to produce paths in $T$ from paths in $\cc{S^z}$.
\begin{lemma} \label{lem:nested}
    Let $u,w$ be simplices of $\cc{S^z}$. If $u \subseteq w$, then $T^u \cap T^w \neq \emptyset$.
\end{lemma}

We now discuss the invariant subtree associated to $G$. Since $G < \mcg{S^z}$ is purely pseudo-Anosov and torsion-free, no element of $G$ fixes any simplex of $\cc{S^z}$. Hence, $G$ acts freely on $T$ as its edges and vertices are $\Gamma_H$-equivariantly identified with simplices of $\cc{S^z}$ in $\Phi^{-1}(\alpha)$. We now define $T^G$ to be the minimal $G$-invariant subtree of $T$; $T^G$ is the union of the axes of loxodromic elements of $G$. A compact fundamental domain for this action can be found by taking the minimal subtree containing: 1) a base vertex $v \in T^G$, and 2) each translate of $v$ by a finite set of generators of $G$. Thus, the action of $G$ on $T^G$ gives $G$ a graph of groups decomposition with trivial vertex and edge groups. From the Bass-Serre structure theorem, we conclude the following lemma.
\begin{lemma} \label{lem:free}
    The group $G$ is free. Moreover, the tree $T^G$ has uniformly finite valence and a free, cocompact $G$-action.
\end{lemma}
The compact graph $T^G/G$ has finitely many edges and vertices. We let $E,V > 0$ be the number of edges and vertices, respectively. Equivalently, $E,V$ is the number of $G$-orbits of edges and vertices in $T^G$, respectively.

Since $G_0$ is a normal, infinite subgroup of $G$, the tree $T^G$ is also the minimal invariant tree of the action of $G_0$ on $T$; $T^G$ is also the union of the axes of the loxodromic elements of $G_0$. Let $\hh_G \subset \HH^2$ denote the convex hull of the limit set of $G$ in $\partial \HH^2$. Using a similar argument to Lemma \ref{lem:multilemma}(4), \cite[Lemma 3.5]{MR4632569} provides a similar statement for $T^G$ and $\hh_G$.
\begin{lemma} \label{lem:TG4}
    A vertex $t \in T$ is a vertex of $T^G$ if and only if $\hh_G \cap \widetilde{Y}_t^\circ \neq \emptyset$.
\end{lemma}

\subsection{Subtree Decomposition}
An important object in this work is the intersection of trees $T^u \cap T^G$ for a simplex $u \subset \cc{S^z}$. We record some useful definitions and results here before proceeding. The lemmas in this section originate from \cite[Section 4.2]{MR4632569}.

\begin{lemma} \label{lem:treeintersectionGeqv}
    For any $g \in G$ and any simplex $u \subset \cc{S^z}$ we have
    $$ g(T^u \cap T^G) = T^G \cap T^{g(u)}. $$
\end{lemma}
\begin{proof}
    Since $g \in G$, we have $g(T^G) = T^G$. Since $K_{g(u)} = gK_ug^{-1}$, by considering the two cases for $T^u$ (minimal invariant or maximal fixed), we see that $g(T^u) = T^{g(u)}$. Therefore, we have
    $$ T^{g(u)} \cap T^G = g(T^u) \cap g(T^G) = g(T^u \cap T^G) $$
\end{proof}

Lemma \ref{lem:multilemma}(3) implies that $T^u \cap T^G \subset T^u$ is uninteresting when $u$ contains non-surviving curves, so we will frequently restrict to $u \subset \ccs{S^z}$ going forwards.

\begin{definition}
    Given a simplex $u \subset \ccs{S^z}$, we say that a vertex $t \in T^u \cap T^G$ is \emph{hull-type} if
    $$ \hh_u \cap \hh_G \cap \widetilde{Y}_t^\circ \neq \emptyset. $$
    Any vertex that is not hull-type is called \emph{parallel-type}.
\end{definition}

The arrangement of hull-type vertices in $T^u \cap T^G$ is nice in that they span a subtree.
\begin{lemma}
    If the set of hull-type vertices is non-empty, then it spans a subtree of $T^u \cap T^G$. That is to say, every vertex of the minimal subtree containing all hull-type vertices is of hull-type.
\end{lemma}
\begin{proof}
    If $t,t' \in T^u \cap T^G$ are hull-type vertices, let $x,x' \in \hh_G \cap \hh_u$ be points with $x \in \widetilde{Y}_t^\circ$ and $x' \in \widetilde{Y}_{t'}^\circ$. By convexity, the geodesic $[x,x'] \subset \HH^2$ is also contained in $\hh_G \cap \hh_u$. Adjusting our equivariant map $\HH^2 \to T$ if necessary, we may assume it maps $[x,x']$ to the geodesic from $t$ to $t'$ in $T^u \cap T^G$. Every vertex of this geodesic is therefore of hull-type.
\end{proof}

The subtree from the previous lemma is called the \emph{hull subtree}, and we denote it $T^\hh$. We call any edge $e \subset T^\hh$ hull-type. Each maximal connected subgraph in the complement of $T^\hh$ in $T^u \cap T^G$ is also a subtree. We call these components \emph{parallel subtrees}, and we will sometimes write $T^{||}$ when referring to one such subtree. We call any edge $e \subset T^{||}$ parallel-type. To avoid casework, we allow the possibility that $T^\hh$ is empty (when there are no hull-type vertices), in which case the entirety of $T^u \cap T^G$ is the unique parallel subtree. On the other hand, if $T^u \cap T^G = T^\hh$, then we consider any parallel subtree to be empty.

The reason for the name parallel-type is justified by the following lemma; see \cite[Lemma 4.10]{MR4632569}.
\begin{lemma} \label{lem:parallelstructure}
    Let $u \subset \ccs{S^z}$ be a multicurve such that $T^u$ has infinite diameter, and let $t_0, \ldots, t_n$ be the vertices of an edge path in $T^u \cap T^G$. Let $e_i$ be the edge from $t_{i-1}$ to $t_i$ and $\widetilde{\alpha}_i$ be the geodesic in $p^{-1}(\alpha)$ that is dual to the edge $e_i$. If each $t_i$ is parallel-type, then there exist geodesics $\delta_u \subset \partial \hh_u$ and $\delta_G \subset \partial \hh_G$ so that
    \begin{itemize}
        \item $\delta_u$ and $\delta_G$ intersect each $\widetilde{\alpha}_i$ transversely
        \item $\delta_u$ and $\delta_G$ do not intersect in $\widetilde{Y}_{t_i}$ for any $i \in \zeroto{n}$
    \end{itemize}
\end{lemma}

\subsection{$G$- and $G_0$-quotients} \label{sec:Gquotient}
Since $G$ acts freely on $T^G$, there is a $G$-equivariant embedding $T^G \to \hh_G$ sending vertices inside the component they are dual to and sending edges to geodesic segments. Taking quotients by $G_0$, we get a surface $\Sigma_0$ with a spine $\sigma_0$
$$ T^G/G_0 = \sigma_0 \subset \Sigma_0 = \hh_G/G_0. $$

Each edge $e$ of $T^G$ intersects exactly one component $\widetilde{\alpha}_e \subset p^{-1}(\alpha)$ and we define $\widetilde{a}_e = \widetilde{\alpha}_e \cap \hh_G$. We write $a_\varepsilon \subset \Sigma_0$ for the image of $\widetilde{a}_e$ in $\Sigma_0$ where $e \subset T^G$ is an edge that projects to an edge $\varepsilon \subset \sigma_0$. Note that for any two edges $\varepsilon, \varepsilon'$ in $\sigma_0$, $a_\varepsilon \cap \varepsilon'$ is empty if $\varepsilon \neq \varepsilon'$, while $a_\varepsilon \cap \varepsilon'$ is a single point if $\varepsilon = \varepsilon'$.

For each vertex $t$ of $T^G$, the intersection $\widetilde{Y}_t \cap \hh_G$ is an even-sided polygon with sides alternating between arcs contained in $p^{-1}(\alpha)$ and those in $\partial \hh_G$. The sides in $p^{-1}(\alpha)$ are precisely the arcs $\widetilde{a}_e$ where $e$ is an edge of $T^G$ incident to $t$. We let $\widetilde{Z}_t \subset \hh_G$ be this polygon corresponding to the vertex $t \in T^G$, and we write $\partial_\alpha \widetilde{Z}_t$ to denote the union of the sides $\widetilde{a}_e$ over all edges $e$ incident to $t$.

Let $\widetilde{p}_0: \HH^2 \to \widetilde{S}_0 = \HH^2/G_0$ be the quotient by $G_0$, which contains $\Sigma_0$ as its convex core, and write $p_0 = \widetilde{p}_0|_{\hh_G}: \hh_G \to \Sigma_0$ for the restriction to $\hh_G$. Let $\eta: \widetilde{S}_0 \to S$ be the associated covering corresponding to $G_0 < \pi_1S$, so that $\eta \circ \widetilde{p}_0 = p$.

Now $\eta^{-1}(\alpha) \cap \Sigma_0$ is a union of the geodesic arcs $a_\varepsilon$ over all edges $\varepsilon$ in $\sigma_0$. The further restriction of $p_0$ to $\widetilde{Z}_t$ is injective on $\widetilde{Z}_t \setminus \partial_\alpha \widetilde{Z}_t$ and maps $\partial_\alpha \widetilde{Z}_t$ into $\eta^{-1}(\alpha)$. For a vertex $\tau \in \sigma_0$, write $Z_\tau = \widetilde{Z}_t$ where $t$ is a vertex of $T^G$ with $p_0(t) = \tau$, and write $Z_\tau \to \Sigma_0$ to denote the restriction of $p_0$. This map is injective, except possibly on the points of $\partial_\alpha \widetilde{Z}_t$. As an abuse of notation, we write $Z_\tau \subset \Sigma_0$ even though it is not necessarily embedded.

Since $G_0$ is a normal subgroup of $G$, we have an action of $G/G_0 \cong H \cong \ZZ^n$ on $\widetilde{S}_0$. Each element $\overline{g} \in G/G_0$ acts as a lift of $\phi(\overline{g}) \in H$ to $\widetilde{S}_0$ which agrees with the lift to $\HH^2$ chosen in Section \ref{sec:gammaaction}. The action of $G/G_0$ on $\widetilde{S}_0$ is free because the action of $G/G_0$ on $\sigma_0$ is free. The action of $G/G_0$ does not preserve $\Sigma_0$, but for $i \in \oneto{n}$ we can find homeomorphisms $\frakh_i: \widetilde{S}_0 \to \widetilde{S}_0$ so that $\frakh_i(\Sigma_0) = \Sigma_0$ and $\frakh_i$ is properly isotopic to the lift of $h_i$ to $\widetilde{S}_0$ via an isotopy that preserves $\eta^{-1}(\alpha)$. If the lift of $h_i$ sends a vertex $\tau$ to a vertex $\tau'$ in $\sigma_0$, then $\frakh_i(Z_\tau) = Z_{\tau'}$; in fact, we use this property to define the isotopy. By further proper isotopy preserving $\eta^{-1}(\alpha)$ and $\Sigma_0$, we may assume $\frakh_i(\sigma_0) = \sigma_0$. The action of $\calH := \gen{\frakh_1, \ldots, \frakh_n}$ on $\Sigma_0$ is a topological covering space action with compact quotient $\Sigma$ that contains $\sigma = \sigma_0/\calH$ as a spine.

The homomorphism $G \to \calH$ given by applying $\Phi_*$, lifting to $\widetilde{S}_0$, then applying the proper isotopy described above descends to an isomorphism $G/G_0 \cong \calH$. Moreover, the projection $T^G \to \sigma_0$ is equivariant with respect to this homomorphism. Via the isomorphism $G/G_0 \cong \calH$, we can identify the quotients $\sigma = \sigma_0/\calH = \sigma_0/(G/G_0) = T^G/G$,

\[\begin{tikzcd}
    T^G \ar{r} \ar{d} & \hh_G \ar[phantom]{r}[marking]{\subset} \ar{d}{p_0} & \HH^2 \ar{d}[swap]{\widetilde{p}_0} \ar{ddr}{p} \\
    \sigma_0 \ar{r} \ar{d} & \Sigma_0 \ar[phantom]{r}[marking]{\subset} \ar{d} & \widetilde{S}_0 \ar{dr}[swap]{\eta} \\
    \sigma \ar{r} & \Sigma && S.
\end{tikzcd}\]

Edges $\varepsilon$ and vertices $\tau$ in $\sigma_0$ inherit some properties from their corresponding $G_0$-orbits of edges and vertices upstairs in $T^G$. Specifically, we say that $\varepsilon$ is $\alpha_j$-dual if any $e$ such that $p_0(e) = \varepsilon$ is $\alpha_j$-dual; thus, $\varepsilon$ is twist if any $e$ with $p_0(e) = \varepsilon$ is twist. Similarly, we say that $\tau$ is $Y_k$-dual if any $t$ such that $p_0(t) = \tau$ is $Y_k$-dual; thus, $\tau$ is pseudo-Anosov if any $t$ with $p_0(t) = \tau$ is pseudo-Anosov. Note that all edges in the same $G/G_0$-orbit share the property of being $\alpha_j$-dual or twist; all vertices in the same $G/G_0$-orbit share the property of being $Y_k$-dual or pseudo-Anosov.

\subsection{Other useful lemmas}
We record a few other useful lemmas here.

Recall that $T = T_\alpha$ is the Bass-Serre tree dual to $p^{-1}(\alpha)$. For any component $\alpha_j \subset \alpha$, we can also consider $T_{\alpha_j}$, the Bass-Serre tree dual to $p^{-1}(\alpha_j)$. $G$ acts on each $T_{\alpha_j}$ for the same reason it acts on $T$. There is a natural collapsing map $c_{\alpha_j}: T \to T_{\alpha_j}$ obtained by collapsing every non-$\alpha_j$-dual edge, and this map descends to a $G$-equivariant map on the $G$-invariant subtrees $c_{\alpha_j}^G: T^G \to T_{\alpha_j}^G$.

\begin{lemma}
    For any component $\alpha_j \subset \alpha$, the collapsing map $c_{\alpha_j}^G: T^G \to T_{\alpha_j}^G$ is a quasi-isometry.
\end{lemma}
\begin{proof}
    The action of $G$ on $T^G$ is free and cellular, so it is properly discontinuous. As noted by Lemma \ref{lem:free}, the action is also cocompact. The Milnor-Schwarz Lemma \cite[Proposition 8.19]{MR1744486} states that the orbit map $G \to T^G$ is a quasi-isometry. By the same argument, $G \to T_{\alpha_j}^G$ is also a quasi-isometry. Since the collapsing map is $G$-equivariant, after choosing appropriate orbit maps, the following diagram commutes,
    \[\begin{tikzcd}[row sep=large]
    & G \ar[swap]{dl}{\text{QI}} \ar{dr}{\text{QI}} \\
    T^G \ar{rr}{c_{\alpha_j}^G} && T_{\alpha_j}^G.
    \end{tikzcd}\]
    Since the other two maps in the diagram are quasi-isometries, $c_{\alpha_j}^G$ must also be a quasi-isometry.
\end{proof}

\begin{corollary} \label{cor:Lbound}
    There is a constant $L_0 > 0$ such that if $\gamma \subset T^G$ is a geodesic edge path with length $\geq L_0$, then for each component $\alpha_j$ of $\alpha$, $\gamma$ must contain at least one edge dual to $\alpha_j$.
\end{corollary}
\begin{proof}
    Fix $\alpha_j$ of $\alpha$ and consider $c_{\alpha_j}: T^G \to T_{\alpha_j}^G$. By the previous lemma, $c_{\alpha_j}$ must be a quasi-isometry for some quasi-isometric constants $(\kappa_j,\lambda_j)$. Suppose that some geodesic edge path $\gamma$ contains no edge dual to $\alpha_j$. Then every edge of $\gamma$ collapses under the collapsing map, so $c_{\alpha_j}(\gamma)$ is just a single vertex. Since both $\gamma$ and $c_{\alpha_j}(\gamma)$ realize the distance between their endpoints, the quasi-isometric inequality gives
    \begin{equation*}
    \begin{aligned}
        \frac{1}{\kappa_j}\len(\gamma) - \lambda_j & \leq \len(c_{\alpha_j}(\gamma)) = 0, \\
        \len(\gamma) & \leq \kappa_j\lambda_j.
    \end{aligned}
    \end{equation*}
    Thus, choosing $L_0 = \max_{j} \kappa_j\lambda_j + 1$ suffices.
\end{proof}
%Alternate approach to corollary can be found in v1.0

The following lemma helps us to pick out useful subsegments of very long geodesics in $T^G$.
\begin{lemma} \label{lem:Gorbitsubsegments}
    For any $N,L > 0$, there is a constant $\mathfrak{L} = \mathfrak{L}(N,L) > 0$ such that if $\gamma \subset T^G$ is a geodesic edge path with length $\geq \mathfrak{L}$, then there must be at least $N+1$ subsegments (pairwise disjoint on their interiors) of length $L$ all contained in the same $G$-orbit.
\end{lemma}
\begin{proof}
    Recall that $E$ is the number of $G$-orbits of edges in $T^G$ or, equivalently, the number of distinct edges in $\sigma = T^G/G$. We will need to take care to distinguish orientation, so we count the edges in $\sigma$ taken with both orientations to yield $2E$ $G$-orbits of oriented edges in $T^G$. We claim that $\mathfrak{L}(N,L) = L (N(2E)^{L} + 1)$ suffices.
    
    Let $\gamma \subset T^G$ be a geodesic edge path with length $\geq L (N(2E)^{L} + 1)$. Orient $\gamma$ to fix a starting endpoint. Let $\gamma_1$ be the subsegment consisting of the first $L$ edges of $\gamma$, and let $\gamma_i$ be the subsegment consisting of the next $L$ edges after $\gamma_{i-1}$. Since $\gamma$ is long enough, we can continue choosing subsegments in this way at least until $\gamma_{N(2E)^{L} + 1}$. We now have a list of $N(2E)^{L} + 1$ subsegments of length $L$ which are pairwise disjoint on their interior.

    There are $2E$ $G$-orbits of oriented edges $e$ in $T^G$. Using this fact, we can deduce that there are at most $(2E)^2$ $G$-orbits of oriented edge paths of length $2$. Inductively, there are at most $(2E)^L$ $G$-orbits of oriented edge paths of length $L$. The pigeonhole principle now implies that among our list of $N(2E)^{L} + 1$ subsegments, some $G$-orbit appears at least $N+1$ times.
\end{proof}

Applying Lemma \ref{lem:Gorbitsubsegments} with $N=1$ and $L=1$ yields the simple observation that a geodesic edge path in $T^G$ of length $\mathfrak{L}(1,1) = 2E+1$ must contain at least one pair of (oriented) edges in the same $G$-orbit. We will be most interested in applying the lemma with some large $N$ and with the constant $L_0$ obtained from Corollary \ref{cor:Lbound}.

\section{Reduction to a Diameter Bound on $p_0(T^u \cap T^G)$}
As mentioned in Section \ref{sec:proofsummary}, the problem of showing $G$ is convex cocompact reduces to proving that $p_0(T^u \cap T^G)$ is uniformly bounded, independent of $u$. We complete this reduction in two steps. The first step shows that a uniform bound on $T^u \cap T^G$ is sufficient to prove $G$ is convex cocompact. The second step shows that a uniform bound on $p_0(T^u \cap T^G)$ is sufficient to prove a uniform bound on $T^u \cap T^G$.

\subsection{First reduction} \label{sec:firstreduction}
The first step is to prove Theorem \ref{thm:main2} while assuming the following the proposition.

\begin{proposition}\label{prop:boundupstairs}
    There exists a constant $D > 0$ such that for any simplex $u \subset \cc{S^z}$ we have $$ \diam(T^u \cap T^G) \leq D, $$ where the diameter is computed in $T^G$.
\end{proposition}

In the $n=1$ case, Leininger and Russell perform this reduction via an argument relating geodesic edge paths in $\cc{S^z}$ with chains of tree intersections in $T^G$. Assuming Proposition \ref{prop:boundupstairs}, these tree intersections are uniformly bounded. Thus, the total diameter of a chain of tree intersections is bounded as a linear function of the length of the geodesic edge path. Since $T^G$ is a geometric model for $G$, this yields the required quasi-isometric constants to show $G$ is convex cocompact.

The same argument works in our case as well, since the rank of $H$ is irrelevant for the proof. The properties of $H$ that are required are that $H$ is reducible -- so that we can define $T_\alpha$ -- and that it is realizable by homeomorphisms -- so that we can define a $\Gamma_H$-action on $\HH^2$. We refer the reader to \cite[Section 4.1]{MR4632569} for the full argument. For a similar arguments using hull intersections which motivated the proof in \cite{MR4632569}, see \cite[Theorem 6.3]{MR2599078} and \cite[Section 7]{MR3314946}.

\subsection{Second reduction} \label{sec:secondreduction}
The second step is to prove Proposition \ref{prop:boundupstairs} while assuming the following proposition.

\begin{proposition} \label{prop:bounddownstairs}
    There exists a constant $D_0 > 0$ such that for any simplex $u \subset \cc{S^z}$ we have $$ \diam(p_0(T^u \cap T^G)) \leq D_0, $$ where the diameter is computed in $\sigma_0$.
\end{proposition}

As described in Section \ref{sec:proofsummary}, we accomplish this task in three substeps. First, we construct a compact subsurface $\Sigma_1 \subset \Sigma_0$ using the bound from Proposition \ref{prop:bounddownstairs}. Next, we reduce to considering only `deep' simplices, which have the property that $p_0(T^u \cap T^G) \subset \Sigma_1$. Finally, we prove Proposition \ref{prop:boundupstairs} for deep simplices.

\subsubsection{Constructing $\sigma_1$ and $\Sigma_1$} \label{sec:constructsigma1}
In order to construct $\Sigma_1 \subset \Sigma_0$, we first construct its spine $\sigma_1 \subset \sigma_0$ then join up the polygons through which the spine travels.

Let $\set{e_1, \ldots, e_E}$ be a set of $\calH$-orbit representatives of edges in $\sigma_0$. If $\delta$ is a boundary component of $\partial\Sigma_0$, then let $\delta'$ denote the minimal length loop in $\sigma_0$ that is freely homotopic in $\Sigma_0$ to $\delta$. Let $D_0$ be the constant obtained from Proposition \ref{prop:bounddownstairs}. Now, let $\sigma_1 \subset \sigma_0$ be a compact, connected subgraph satisfying the following properties.
\begin{itemize}
    \item $\set{e_1, \ldots, e_E} \subset \sigma_1$
    \item for each $e_i \in \set{e_1, \ldots, e_E}$, the distance between $e_i$ and any point in $\sigma_0 \setminus \sigma_1$ is at least $D_0 + 2$
    \item for each component $\delta$ of $\partial\Sigma_0$, $\exists \frakh \in \calH$ such that $\frakh(\delta') \subset \sigma_1$
    \item $\sigma_1$ contains no valence $1$ vertices
    \item $\sigma_1$ contains all edges between with endpoints in its vertex set
\end{itemize}
Such a subgraph can be constructed as follows. First, let $\sigma_1$ be the closed $(D_0 + 2)$-neighborhood of $\set{e_1, \ldots, e_E}$. Since $\Sigma \cong \Sigma_0 / \calH$ is a compact surface, there are only finitely many $\calH$-orbits of boundary components of $\Sigma_0$. For each $\calH$-orbit, choose a representative $\delta$, and add $\delta'$ to $\sigma_1$. If the result is not connected, add paths between disjoint components to produce a connected subgraph. The current construction is compact, connected, and satisfies the first three properties. Now consider the complement $\sigma_0 \setminus \sigma_1$. Since $\sigma_0$ is one-ended and $\sigma_1$ is compact, $\sigma_0 \setminus \sigma_1$ consists of finitely many bounded components and a single unbounded component that is a neighborhood of the end. Add the finitely many bounded components to $\sigma_1$ so that now $\sigma_0 \setminus \sigma_1$ is just the unbounded component. Add in all edges between vertices already in $\sigma_1$. If $t \in \sigma_1$ has valence $1$ in $\sigma_1$, then it must be adjacent to some $t' \in \sigma_0 \setminus \sigma_1$ since $\sigma_0$ has no valence $1$ vertices. Now for any $s' \in \sigma_0 \setminus \sigma_1$ adjacent to an $s \in \sigma_1$ with $s \neq t$, there is a path from $t'$ to $s'$ in $\sigma_0 \setminus \sigma_1$. Adding the path from $t$ to $s$ which includes this path from $t'$ to $s'$ to $\sigma_1$ makes $t$ no longer valence $1$ and adds no new valence $1$ vertex. Repeat the process until $\sigma_1$ has no valence $1$ vertices. Finally, again add in all edges between vertices already in $\sigma_1$.

We now move to describing $\Sigma_1$ and related objects. Let $\Sigma_1$ be the compact subsurface of $\Sigma_0$ defined by

\[
\Sigma_1 = \bigcup_{\tau \in \sigma_1^{(0)}} Z_\tau.
\]

Just as $\sigma_0$ is the spine of $\Sigma_0$, $\sigma_1$ is the spine of $\Sigma_1$. Let $G_1 < G_0$ be the image of $\pi_1\Sigma_1$ in $G_0 = \pi_1\Sigma_0$. Equivalently, $G_1 < G_0$ is the image of $\pi_1\sigma_1$ in $G_0 = \pi_1\sigma_0$. Note that $G_1$ is finitely generated since $\Sigma_1$ is compact, and $G_1 < G$ is purely pseudo-Anosov. Let $\hh_{G_1} \subset \HH^2$ be the convex hull of the limit set of $G_1$.

A similar union of polygons is also useful upstairs. Let $\widetilde{\sigma}_1 \subset T^G$ be the component of $p_0^{-1}(\sigma_1)$ that is $G_1$-invariant. Let $\hh^G_1$ be the closed subspace of $\hh_G$ defined by

\[
\mathfrak{H}^G_1 = \bigcup_{t \in \widetilde{\sigma}_1^{(0)}} \widetilde{Z}_t.
\]

Note that $\hh^G_1$ is the minimal closed, convex, $G_1$-invariant subspace of $\hh_G$ that projects to $\Sigma_1$,

\[\begin{tikzcd}[column sep=tiny,row sep=tiny]
& {T^G} \ar{rrr} \ar{ddd} &&& {\hh_G} \ar{ddd} \\
{\widetilde{\sigma}_1} \ar[phantom, sloped]{ur}[marking]{\subset} \ar[dashed]{rrr} &&& {\mathfrak{H}^G_1} \ar[phantom, sloped]{ur}[marking]{\subset} && {\hh_{G_1}} \\
\\
& {\sigma_0} \ar{rrr} &&& {\Sigma_0} \\
{\sigma_1} \ar[phantom, sloped]{ur}[marking]{\subset} \ar[dashed]{uuu} \ar[dashed]{rrr} &&& {\Sigma_1} \ar[phantom, sloped]{ur}[marking]{\subset} \ar[dashed, bend right=45]{uuurr} && .
\end{tikzcd}\]

Let $R := \max_{\tau \in \sigma_1^{(0)}} \diam(Z_\tau)$. The two subspaces $\hh^G_1$ and $\hh_{G_1}$ are related by the following lemma proved in \cite[Section 4.2.2]{MR4632569}.

\begin{lemma} \label{lem:G1hull}
    $\mathfrak{H}^G_1 \subset N_R(\hh_{G_1})$
\end{lemma}
\begin{proof}
    Since $\sigma_1$ has no valence $1$ vertices, one can produce a closed loop with no backtracking that visits every vertex of $\sigma_1$. Let $\gamma$ be the geodesic representative in $\Sigma_0$ of any such loop. $\gamma$ intersects $Z_\tau$ for every $\tau \in \sigma_1^{(0)}$. Now, for every $t \in \widetilde{\sigma}_1^{(0)}$, there is a geodesic in $p_0^{-1}(\gamma)$ that is invariant by an infinite cyclic subgroup of $G_1$ and intersects $\widetilde{Z}_t$. Since any such geodesic lies in $\hh_{G_1}$, all $\widetilde{Z}_t$ with $t \in \widetilde{\sigma}_1^{(0)}$ are contained in $N_R(\hh_{G_1})$, completing the proof.
\end{proof}

\subsubsection{Reduction to Deep Simplices}
We now define and reduce to deep simplices. We say that a simplex $u \subset \ccs{S^z}$ is \emph{deep} if $N_2(p_0(T^u \cap T^G)) \subset \sigma_1$, where $N_2(\cdot)$ is the $2$-neighborhood in $\sigma_0$. The following lemma is key to the reduction and is slightly modified from \cite[Lemma 4.5]{MR4632569}.

\begin{lemma} \label{lem:deepsimplex}
    For any simplex $u \subset \cc{S^z}$, there exists a $g \in G$ such that $g \cdot u$ is a deep simplex; i.e.
    $$ N_2(p_0(T^{g \cdot u} \cap T^G)) \subset \sigma_1. $$
\end{lemma}
\begin{proof}
    For any $u$, there is some $\frakh \in \calH$ and one of the $\calH$-orbit representatives $e_i$ such that
    $$ e_i \subset \frakh(p_0(T^u \cap T^G)). $$
    Then $N_2(\frakh(p_0(T^u \cap T^G))) \subset \sigma_1$ by the second property of $\sigma_1$. Let $g \in G$ be any element that maps to $\frakh$ under the homomorphism $G \to \calH$. Since $p_0: T^G \to \sigma_0$ is equivariant with respect to this homomorphism, Lemma \ref{lem:treeintersectionGeqv} implies
    $$ p_0(g(T^u \cap T^G)) = \frakh(p_0(T^u \cap T^G)). $$
    This completes the proof.
\end{proof}

\begin{remark}\label{rmk:deep}
    The importance of Lemma \ref{lem:deepsimplex} is the following. If $u$ is not a deep simplex, then Lemma \ref{lem:deepsimplex} says we can find a $g \in G$ such that $g \cdot u$ is a deep simplex. Further, Lemma \ref{lem:treeintersectionGeqv} implies
    $$ \diam(T^u \cap T^G) = \diam(T^{g \cdot u} \cap T^G). $$
    Thus, a uniform bound for deep simplices will suffice for all simplices.
\end{remark}

\subsubsection{A Diameter Bound on $T^u \cap T^G$}
Having reduced to considering only deep simplices, we move to the proving Proposition \ref{prop:boundupstairs}. The strategy is to first obtain separate bounds for the hull and parallel subtrees. Then, we combine the two bounds to produce a bound for $T^u \cap T^G$. Leininger and Russell carried out this strategy in the $n=1$ case, and the same argument will work here. We give an outline of the proofs here but refer the reader to \cite[Section 4.2]{MR4632569} for detailed arguments.

The following lemma proves a bound on the diameters of the hull subtree. See \cite[Lemma 4.8]{MR4632569} for a detailed proof.
\begin{lemma} \label{lem:Dhull}
    There exists a constant $D^\hh > 0$ such that for any deep simplex $u \subset \cc{S^z}$, the diameter of the hull subtree $T^\hh$ is at most $D^\mathfrak{H}$.
\end{lemma}
\begin{proof}[Proof outline]
    Under our equivariant Lipschitz map $\HH^2 \to T$ chosen in Section \ref{sec:gammaaction}, the hull intersection $\hh_u \cap \hh_G$ maps to a set of Hausdorff distance at most $\frac{1}{2}$ from $T^\hh$. Thus, it suffices to bound $\hh_u \cap \hh_G$. Using the fact that $u$ is deep, Lemma \ref{lem:multilemma}, and Lemma \ref{lem:G1hull}, we find that
    $$ \hh_u \cap \hh_G = \hh_u \cap \hh^G_1 \subset \hh_u \cap N_R(\hh_{G_1}). $$
    Finally, since $G_1 < \pi_1S$ is finitely generated and purely pseudo-Anosov, \cite[Corollary 5.2]{MR2599078} gives a uniform bound on $\hh_u \cap N_R(\hh_{G_1})$.
\end{proof}

The following lemma proves a bound on the diameters of the parallel subtrees. See \cite[Lemma 4.12]{MR4632569} for a detailed proof.
\begin{lemma} \label{lem:Dparallel}
    There exists a constant $D^{||} > 0$ such that for any deep simplex $u \subset \cc{S^z}$, the diameter of any parallel subtree $T^{||}$ is at most $D^{||}$.
\end{lemma}
\begin{proof}[Proof outline]
    Let $t_0, \ldots, t_n$ be the vertices of a geodesic edge path in a parallel subtree $T^{||}$. From Lemma \ref{lem:parallelstructure}, we obtain geodesics $\delta_u$ and $\delta_G$ which form a convex hyperbolic quadrilateral with the geodesics $\widetilde{\alpha}_1$ and $\widetilde{\alpha}_n$. Let $\delta_G'$ denote the side of the quadrilateral contained in $\delta_G$. Using the fact that $u$ is deep, we find $p_0(\delta_G') \subset \partial\Sigma_0 \cap \partial\Sigma_1$. Each component of $\partial\Sigma_0 \cap \partial\Sigma_1$ is a geodesic arc or closed curve, and the compactness of $\Sigma_1$ gives a bound on the length each component. If $p_0(\delta_G')$ is contained in an arc, the bound on the arc gives a bound on $n$, as desired. Otherwise, $p_0(\delta_G')$ is contained in a geodesic curve $c$, and in fact, all of $\delta_G$ maps onto $c$. Now $c$ corresponds to an element of $G_1 = \pi_1\Sigma_1 < \pi_1S$ which is represented in $\pi_1S < \mcg{S^z,z}$ by the point push of $z$ along a geodesic curve $\gamma^{-1}$. Note $\gamma = \eta(c)$. Since $G_1$ is finitely generated and purely pseudo-Anosov, Theorem \ref{thm:kra} states that $\gamma$ is a filling curve on $S$. On the other hand, $p(\delta_u) \subset v = \Phi(u)$ is a simple closed curve. Using hyperbolic geometry, we produce a bound on how long a lift of a simple closed curve -- such as $\delta_u$ -- can travel close to a lift of $\gamma$ -- such as $\delta_G$.
\end{proof}

Finally, we combine the bounds to prove Proposition \ref{prop:boundupstairs}.
\begin{proof}[Proof of Proposition \ref{prop:boundupstairs}]
    We will show that the proposition holds for $D = D^\mathfrak{H} + 2 D^{||} + 2$.

    Let $t,t' \in T^u \cap T^G$ be any two vertices. Let $\gamma \subset T^u \cap T^G$ be the geodesic edge path between $t$ and $t'$. $\gamma$ decomposes into at most 5 segments: 1 segment contained in the hull subtree, 2 segments contained in parallel subtrees, and 2 edges joining the segments in the parallel subtrees with the segment in the hull subtree. It follows from Lemma \ref{lem:Dhull} that the length of the hull subtree segment is at most $D^\hh$, and it follows from Lemma \ref{lem:Dparallel} that the lengths of the parallel subtree segments are at most $D^{||}$. Thus, $l(\gamma) \leq D^\mathfrak{H} + 2 D^{||} + 2 = D$. Since $t,t'$ were arbitrary, we have that $\diam(T^u \cap T^G) \leq D$.
\end{proof}

\section{A Diameter Bound on $p_0(T^u \cap T^G)$} \label{sec:main}
In the previous section, we reduced the proof of Theorem \ref{thm:main2} to proving Proposition \ref{prop:bounddownstairs}. We now proceed to find a uniform bound on the diameter of $p_0(T^u \cap T^G)$, independent of $u$. We again employ the strategy of first obtaining separate bounds for $p_0(T^\hh)$ and $p_0(T^{||})$. Rather than compute these bounds directly in $\sigma_0$, it is more convenient to produce a bound in a space quasi-isometric to $\sigma_0$.

\subsection{Constructing $\sigma_{JK}$}
We recall the notation from Section \ref{sec:Haction} and the injective homomorphism $V: H \to \ZZ^{m+l}$. Let $\calM$ be the $(m+l)$-by-$n$ matrix whose columns are the vectors $V(h_i)$. The rows of $\calM$ correspond to the generators $\set{\tau_1, \ldots, \tau_m, \widehat{\psi}_1, \ldots, \widehat{\psi}_l}$. Continuing with our notation from Remark \ref{rmk:k'notation}, denote each of the first $m$ rows by $R_j$ for $j \in \oneto{m}$, and denote each of the last $l$ rows by $R_{k'}$ for $k \in \oneto{l}$. One way to interpret $\calM$ is as follows. If $h = h_1^{a_1} \circ \cdots \circ h_n^{a_n}$, then
\begin{equation} \label{eq:matrixM}
    \calM \begin{bmatrix} a_1 \\ \vdots \\ a_n \end{bmatrix} = V(h).
\end{equation}

The $n$ column vectors of $\calM$ are linearly independent since $V$ is injective, so there must be some collection of $n$ linearly independent row vectors. Fix one such collection $\calR$. Let $J \subset \oneto{m}$ be the indexing set containing the indices $j$ for which $R_j \in \calR$. Let $K \subset \oneto{l}$ be the indexing set containing the indices $k$ for which $R_{k'} \in \calR$. Let $\calN$ be the $n$-by-$n$ matrix whose rows are the vectors in $\calR$ ordered by increasing index. The square matrix $\calN$ is nonsingular and is obtained by deleting all but $n$ linearly independent rows of $\calM$.

No row vector in $\calR$ (or $\calN$) can be the zero vector. As a result, $\alpha_j$ is twist for all $j \in J$, and $Y_k$ is pseudo-Anosov for all $k \in K$. Recall from Section \ref{sec:Haction} the groups $H_j = \ker \widehat{\rho}_j$ and $H_{k'} = \ker \widehat{\rho}_{k'}$ and their isomorphic images $\overline{G}_j$ and $\overline{G}_{k'}$ in $G/G_0$. Also recall that if $\alpha_j$ is twist, then $(G/G_0)/\overline{G}_j \cong \ZZ$; if $Y_k$ is pseudo-Anosov, then $(G/G_0)/\overline{G}_{k'} \cong \ZZ$. The fact that $\calN$ is nonsingular implies
\begin{equation} \label{eq:Hintersection}
    \bigcap_{j \in J} H_j \cap \bigcap_{k \in K} H_{k'} = \{ \id_{H} \} \quad \text{and} \quad \bigcap_{j \in J} \overline{G}_j \cap \bigcap_{k \in K} \overline{G}_{k'} = \{ \id_{G/G_0} \}
\end{equation}

For each $j \in J$ and $k \in K$, and using the cellular, geometric action of $G/G_0$ on $\sigma_0$, we define quotient graphs $\sigma_j = \sigma_0 / \overline{G}_j$ and $\sigma_{k'} = \sigma_0 / \overline{G}_{k'}$ and let $p_j: \sigma_0 \to \sigma_j$ and $p_{k'}: \sigma_0 \to \sigma_{k'}$ denote the quotient maps. Let $\sigma_{JK}$ be the product cube complex $\prod_{j \in J} \sigma_j \times \prod_{k \in K} \sigma_{k'}$, and let $p_{JK}: \sigma_0 \to \sigma_{JK}$ be the product of quotient maps $\prod_{j \in J} p_j \times \prod_{k \in K} p_{k'}$.

Each quotient $\sigma_j$ admits a cellular, geometric action by $(G/G_0)/\overline{G}_j \cong \ZZ$. Each $\sigma_j$ also admits an action by $G/G_0$ which is no longer free, and the kernel of this action is exactly $\overline{G}_j$. The quotient maps $p_j: \sigma_0 \to \sigma_j$ are equivariant with respect to this action by $G/G_0$. The same holds for $\sigma_{k'}$ and $\overline{G}_{k'}$ and $p_{k'}$. Consider the diagonal action of $G/G_0$ on the product $\sigma_{JK}$. This action is free because the kernels of the action in each factor are the groups $\overline{G}_j$ and $\overline{G}_{k'}$, and their intersection is trivial (Equation \ref{eq:Hintersection}). The product map $p_{JK}$ is also $G/G_0$-equivariant.

\begin{lemma} \label{lem:qitosigma0}
    The map $p_{JK}: \sigma_0 \to \sigma_{JK}$ is a quasi-isometry.
\end{lemma}
\begin{proof}    
    The action of $G/G_0$ on $\sigma_0$ is properly discontinuous and cocompact, so the Milnor-Schwarz Lemma \cite[Proposition 8.19]{MR1744486} says that the orbit map $G/G_0 \to \sigma_0$ is a quasi-isometry.

    The action of $G/G_0$ on $\sigma_{JK}$ is free and cellular, and thus properly discontinuous. Further, $G/G_0 \cong \ZZ^n$ acts cocompactly on each factor $\sigma_j$ or $\sigma_k$, and there are exactly $n$ factors. Since the kernel of the action on each factor is given by $H_j$ or $H_k$, and the intersection of all $\overline{G}_j$ and $\overline{G}_{k'}$ is the trivial group (Equation \ref{eq:Hintersection}), the action of $G/G_0$ on $\sigma_{JK}$ is also cocompact. The Milnor-Schwarz Lemma again tells us the orbit map $G/G_0 \to \sigma_{JK}$ is also a quasi-isometry. Since $p_{JK}$ is $G/G_0$-equivariant, after choosing appropriate orbit maps, the following diagram below commutes,

    \[\begin{tikzcd}[row sep=2.5em]
    & G/G_0 \ar[swap]{dl}{\text{QI}} \ar{dr}{\text{QI}} \\
    \sigma_0 \ar{rr}{p_{JK}} && \sigma_{JK}.
    \end{tikzcd}\]

    Since the other two maps in the diagram are quasi-isometries, $p_{JK}$ must also be a quasi-isometry.
\end{proof}

\subsection{Edge and vertex decorations} \label{sec:decorations}
To each edge $e$ and vertex $t$ of $T^G$, we will assign a bounded diameter subset $\dec{e} \subset \ac{A_e}$ and $\dec{t} \subset \acc{Y_t}$, respectively. These are called \emph{decorations} of the edges and vertices.

For each edge $e$ of $T^G$, there are exactly two geodesics in $\partial \hh_G$ that non-trivially intersect $\widetilde{\alpha}_e$. Define $\dec{e}$ to be the union of the images of these two geodesics under the covering map $\HH^2 \to A_e$. If $e$ and $e'$ are edges of $T^G$ that are in the same $G_0$-orbit, then $A_e = A_{e'}$ and $\dec{e} = \dec{e'}$ because $G_0$ preserves $\hh_G$.

For each vertex $t$ in $T^G$, each geodesic arc $\widetilde{\gamma}$ in $\widetilde{Z}_t \subset \widetilde{Y}_t$ with endpoints in $\partial_\alpha \widetilde{Z}_t$ projects to a geodesic path $\gamma$ in $Y_t$. For each such path $\gamma$, we consider the self-intersection number $\II(\gamma)$, which is the minimum number of double points of self-intersection over all representatives of the homotopy class rel endpoints (which is realized by the unique geodesic representative orthogonal to the boundary). For each $t$, there are only finitely many homotopy classes of such arcs $\gamma_1, \ldots, \gamma_{r(t)}$, and we define
$$ \dec{t} = \{ \beta \in \acc{Y_t} \mid i(\beta, \gamma_j) \leq 2\II(\gamma_j) \text{ for some } j \in \set{1,\ldots,r(t)} \}. $$

Note that by taking a representative of $\gamma_j$ with only double points of self-intersection realizing $\II(\gamma_j)$, we can construct an arc $\beta_j$ in $Y_t$ from surgery on these self-intersection points, and then pushing off, so that $i(\beta_j, \gamma_j) \leq 2\II(\gamma_j)$. In particular, $\dec{t} \neq \emptyset$. Moreover, any $\beta$ with $i(\beta, \gamma_j) \leq 2\II(\gamma_j)$ also satisfies $i(\beta, \beta_j) \leq 2\II(\gamma_j)$ since $\beta_j$ is constructed from arcs of $\gamma_j$. Since distance is bounded by a function of intersection number (see \cite{MR1714338}), it follows that $\dec{t}$ has finite diameter in $\acc{Y_t}$. As with the edge decorations, if $t$ and $t'$ are vertices in the same $G_0$-orbit, then $\dec{t} = \dec{t'}$.

The following lemma describes how these decorations behave under arbitrary elements of $G$.
\begin{lemma} \label{lem:decaction}
    For any edge $e$ or vertex $t$ of $T^G$ and $g \in G$, we have
    \begin{gather*}
        \dec{g(e)} = \Phi_*(g)(\dec{e}), \\
        \dec{g(t)} = \Phi_*(g)(\dec{t}).
    \end{gather*}
\end{lemma}
\begin{proof}
    Given $g \in G$, $g$ does not necessarily preserve $\partial \hh_G$. On the other hand, $g$ does map each geodesic of $\partial \hh_G$ to a bi-infinite path that is homotopic, rel the ideal endpoints, to a geodesic in $\partial \hh_G$. This is because geodesics are completely determined by the components of $p^{-1}(\alpha)$ that are intersected. Since $g$ descends to a homeomorphism isotopic to the lift of $\Phi_*(g)$ on each $A_e = A_{g(e)}$, the first equation follows.

    For the second equation, let $\widetilde{\gamma} \subset \widetilde{Z}_t$ be any geodesic arc with endpoints in $\partial_\alpha \widetilde{Z}_t$ and let $\gamma$ be its image path in $Y_t$. We observe that $g$ descends to the restriction of $\Phi_*(g)$ to $Y_t = Y_{g(t)}$, and so maps $\gamma$ to a path $\Phi_*(g)(\gamma)$ homotopic rel $\partial Y_{g(t)}$ to the image of a geodesic in $\widetilde{Z}_{g(t)}$. Therefore, the restriction of $\Phi_*(g)$ to $Y_t$ maps the finite set of homotopy classes of paths defining $\dec{t}$ to those defining $\dec{g(t)}$, and hence sends $\dec{t}$ to $\dec{g(t)}$.
\end{proof}

\begin{corollary} \label{cor:B0}
    There exists a constant $B_0 > 0$ so that $$ \diam(\dec{e}),\diam(\dec{t}) \leq B_0, $$ for all vertices $t$ and edges $e$.
\end{corollary}
\begin{proof}
    For any $g \in G$, $\Phi_*(G)$ acts by simplicial automorphisms on $\ac{A_e}$ and $\acc{Y_t}$ for every edge $e$ and vertex $t$. Combining this with Lemma \ref{lem:decaction}, we have
    \begin{gather*}
        \diam(\dec{g(e)}) = \diam(\Phi_*(g)(\dec{e})) = \diam(\dec{e}), \\
        \diam(\dec{g(t)}) = \diam(\Phi_*(g)(\dec{t})) = \diam(\dec{t}).
    \end{gather*}
    In other words, the diameters of decorations are shared among $G$-orbits of edges and vertices. Since there are only finitely many $G$-orbits of edges and vertices, it suffices to take $B_0$ to be the maximum diameter of $\dec{e}$ and $\dec{t}$ taken over a finite set of $G$-orbit representatives of edges $e$ and vertices $t$.
\end{proof}

Since $\dec{e} = \dec{e'}$ and $\dec{t} = \dec{t'}$ for $e,e'$ or $t,t'$ in the same $G_0$-orbit, these decorations on edges and vertices of $T^G$ descend to decorations on the edges and vertices of $\sigma_0 = T^G/G_0$. We denote these by $\dec{\varepsilon}$ and $\dec{\tau}$ for an edge $\varepsilon$ or vertex $\tau$ in $\sigma_0$. The action of $G/G_0$ on $\sigma_0$ induces an action on the decorations. As a result of Lemma \ref{lem:decaction}, this action satisfies the following analogous formulae,
\begin{gather*}
    \dec{\overline{g}(e)} = \phi(g)(\dec{e}), \\
    \dec{\overline{g}(t)} = \phi(g)(\dec{t}),
\end{gather*}
for every edge $\varepsilon$ and vertex $\tau$ in $\sigma_0$ and every $\overline{g} \in G/G_0$.

\subsection{Subsurface projections} \label{sec:ssp}
Given a multicurve $v \subset \cc{S}$, Masur and Minsky defined a projection of $v$ to the arc and curve graphs of subsurfaces and annular covers of $S$ \cite{MR1791145}. We describe these projections in the special cases of $A_e$ and $Y_t$.

For each vertex $t \in T^G$, the multicurve $v$ intersects $Y_t$ in a collection of disjoint curves and arcs, producing a (possibly empty) simplex of $\acc{Y_t}$. Let $\ssp{t}{v}$ be this simplex. We observe that $\ssp{t}{v}$ is precisely the set of essential arcs and curves that are in the image of $p^{-1}(v) \cap \widetilde{Y}_t$ under the covering map $\widetilde{Y}_t \to Y_t$. Since $Y_t = Y_{t'}$ if $t,t'$ are in the same $G$-orbit, we have $\ssp{t}{v} = \ssp{t'}{v}$ in this case. Similarly, for any two $Y_k$-dual vertices $t,t'$, we have $Y_t = Y_{t'}$ and $\ssp{t}{v} = \ssp{t'}{v}$. Thus, continuing with our previous notation we define $\ssp{k'}{v} = \ssp{t}{v}$ for any $Y_k$-dual $t$.

For each edge $e \subset T^G$, we define $\ssp{e}{v} \subset \ac{A_e}$ to be the set of essential arcs in the preimage of $v$ under the covering map $A_e \to S$. As in the case of $\ssp{t}{v}$, we note that $\ssp{e}{v}$ is precisely the essential arcs in the image of $p^{-1}(v)$ under the covering map $\HH^2 \to A_e$. Since $v$ is a collection of disjoint curves, $\ssp{e}{v}$ is a simplex of $\ac{A_e}$. Since the core curve of $A_e$ is a lift of a curve $\alpha_j \subset \alpha$, we have $\ssp{e}{v} \neq \emptyset$ if and only if $i(v,\alpha_j) \neq 0$. Since $A_e = A_{e'}$ if $e,e'$ are in the same $G$-orbit, we have $\ssp{e}{v} = \ssp{e'}{v}$ in this case. Similarly, for any two $\alpha_j$-dual edges $e,e'$, we have $A_e = A_{e'}$ and $\ssp{e}{v} = \ssp{e'}{v}$. Thus, we define $\ssp{j}{v} = \ssp{e}{v}$ for any $\alpha_j$-dual $e$.

We also define $\calS(v) = \set{\alpha_j \mid \ssp{j}{v} \neq \emptyset} \cup \set{Y_k \mid \ssp{k'}{v} \neq \emptyset}.$ In other words, $\calS(v)$ is the set of reducing curves and subsurfaces which intersect $v$.

Since $A_e$ and $Y_t$ are determined by the $G$-orbit of the edge or vertex, we can define projection for vertices and edges in $\sigma_0$ by
\begin{gather*}
    \ssp{\varepsilon}{v} = \ssp{e}{v}, \\
    \ssp{\tau}{v} = \ssp{t}{v},
\end{gather*}
where $\varepsilon = p_0(e)$ and $\tau = p_0(t)$. If $\varepsilon,\varepsilon'$ are edges in the same $G/G_0$-orbit, then $\ssp{\varepsilon}{v} = \ssp{\varepsilon'}{v}$; if $\tau,\tau'$ are vertices in the same $G/G_0$-orbit, then $\ssp{\tau}{v} = \ssp{\tau'}{v}$.

Given an edge $e$ or vertex $t$ of $T^G$, we let $d(\dec{e}, \ssp{e}{v})$ and $d(\dec{t}, \ssp{t}{v})$ denote the diameter of $\dec{e} \cup \ssp{e}{v}$ and $\dec{t} \cup \ssp{t}{v}$ in $\ac{A_e}$ and $\acc{Y_t}$, respectively. We also use similar notation for edges $\varepsilon$ and vertices $\tau$ in $\sigma_0$.

Given $B > 0$, we define in $T^G$ the following sets
\[
\begin{aligned}
    \widetilde{\mathcal{E}}(v,B) & = \{ e \subset T^G \mid e \text{ is twist}, \ssp{e}{v} \neq \emptyset, d(\dec{e}, \ssp{e}{v}) \leq B \}, \\
    \widetilde{\mathcal{V}}(v,B) & = \{ t \in T^G \mid t \text{ is pseudo-Anosov}, \ssp{t}{v} \neq \emptyset, d(\dec{t}, \ssp{t}{v}) \leq B \}.
\end{aligned}
\]

Now since being twist/pseudo-Anosov, $\ssp{e}{v}$/$\ssp{t}{v}$, and $\dec{e}$/$\dec{t}$ are shared by $G_0$-orbits of edges/vertices, we can also define in $\sigma_0$
\[
\begin{aligned}
    \mathcal{E}(v,B) & = \{ \varepsilon \subset \sigma_0 \mid \varepsilon \text{ is twist}, \ssp{\varepsilon}{v} \neq \emptyset, d(\dec{\varepsilon}, \ssp{\varepsilon}{v}) \leq B \}, \\
    \mathcal{V}(v,B) & = \{ \tau \in \sigma_0 \mid \tau \text{ is pseudo-Anosov}, \ssp{\tau}{v} \neq \emptyset, d(\dec{\tau}, \ssp{\tau}{v}) \leq B \},
\end{aligned}
\]
and notice
\[
\begin{aligned}
    \mathcal{E}(v,B) & = p_0(\widetilde{\mathcal{E}}(v,B)), \\
    \mathcal{V}(v,B) & = p_0(\widetilde{\mathcal{V}}(v,B)).
\end{aligned}
\]

In the $n=1$ case, $\sigma_0$ is quasi-isometric to $\ZZ$, and Leininger and Russell show that $\mathcal{E}(v,B)$ and $\mathcal{V}(v,B)$ are bounded sets, akin to the unions of intervals on a line. Unfortunately in the $n \geq 2$ case, $\sigma_0$ is quasi-isometric to $\ZZ^n$, and $\mathcal{E}(v,B)$ and $\mathcal{V}(v,B)$ are not bounded, but rather appear as something like ``strips in a plane'' in the case $n=2$. To handle this more complicated situation, we decompose $\mathcal{E}(v,B)$ and $\mathcal{V}(v,B)$ into appropriate pieces, then prove that the images of these pieces in appropriate quotients are again bounded.

Define
\[
\begin{aligned}
    \mathcal{E}_j(v,B) & = \{ \varepsilon \subset \mathcal{E}(v,B) \mid \varepsilon \text{ is } \alpha_j \text{-dual} \}, \\
    \mathcal{V}_k(v,B) & = \{ \tau \in \mathcal{V}(v,B) \mid \tau \text{ is } Y_k \text{-dual} \},
\end{aligned}
\]
and notice
\[
\begin{aligned}
    \mathcal{E}(v,B) & = \bigcup_{j=1}^m \mathcal{E}_j(v,B), \\
    \mathcal{V}(v,B) & = \bigcup_{k=1}^l \mathcal{V}_k(v,B).
\end{aligned}
\]

We similarly partition $\widetilde{\mathcal{E}}(v,B)$ and $\widetilde{\mathcal{V}}(v,B)$ into pieces $\widetilde{\mathcal{E}}_j(v,B)$ and $\widetilde{\mathcal{V}}_k(v,B)$. We will sometimes suppress the notation $(v,B)$ when the context is clear.

Recall that if an edge $\varepsilon$ is $\alpha_j$-dual for $j \in J$, $\varepsilon$ is always twist; if a vertex $\tau$ is $Y_k$-dual for $k \in K$, $\tau$ is always pseudo-Anosov. Further recall that $E,V$ is the number of $G$-orbits of edges and vertices in $T^G$, respectively. Let $E_j$ be the number of $G$-orbits of $\alpha_j$-dual edges in $T^G$; equivalently, $E_j$ is the number of $G/G_0$-orbits of $\alpha_j$-dual edges in $\sigma_0$. Let $V_k$ be the number of $G$-orbits of $Y_k$-dual vertices in $T^G$; equivalently, $V_k$ is the number of $G/G_0$-orbits of $Y_k$-dual vertices in $\sigma_0$. Note that $E = \sum_{j=1}^m E_j$ and $V = \sum_{k=1}^l V_k$.

\begin{lemma} \label{lem:pjkbounded}
    For any $B > 0$, there exists $M > 0$, such that the following hold for each simplex $v \subset \cc{S}$.
    \begin{enumerate}
        \item For each $j \in J$, $p_j(\mathcal{E}_j(v,B))$ is a union of at most $E_j$ sets of diameter at most $M$.
        \item For each $k \in K$, $p_{k'}(\mathcal{V}_k(v,B))$ is a union of at most $V_k$ sets of diameter at most $M$.
    \end{enumerate}
\end{lemma}
\begin{proof}
    Fix $B > 0$ and a multicurve $v \subset \cc{S}$. We first prove (1).
    
    Fix $j \in J$ and an $\alpha_j$-dual edge $\varepsilon$ in $\sigma_0$. Choose $\overline{g} \in G/G_0$ such that the coset $\overline{g} \overline{G}_j \in (G/G_0)/\overline{G}_j$ maps to a generator of $\ZZ$ via the isomorphism $(G/G_0)/\overline{G}_j \cong \ZZ$. Then, $G/G_0 \cdot \varepsilon$ can be partitioned into coset orbits.
    \begin{equation*}
        G/G_0 \cdot \varepsilon = \bigsqcup_{n \in \ZZ} \overline{g}^n\overline{G}_j \cdot \varepsilon.
    \end{equation*}
    We claim that there is some finite interval $I_\varepsilon \subset \ZZ$ such that
    \begin{equation} \label{eq:Ejsetinterval}
        (G/G_0 \cdot \varepsilon) \cap \mathcal{E}_j(v,B) \subset \bigsqcup_{n \in I_\varepsilon} \overline{g}^n\overline{G}_j \cdot \varepsilon.
    \end{equation}
    The interval $I_\varepsilon$ can be computed as follows.
    
    Recall that $\ssp{\varepsilon'}{v} = \ssp{\varepsilon}{v}$ for all edges $\varepsilon'$ in $G/G_0 \cdot \varepsilon$. If $\ssp{\varepsilon}{v} = \emptyset$, then $(G/G_0 \cdot \varepsilon) \cap \mathcal{E}_j(v,B) = \emptyset$, and the containment \ref{eq:Ejsetinterval} holds for the empty interval. Suppose $\ssp{\varepsilon}{v} \neq \emptyset$. Since every $\alpha_j$-dual edge is twist, an arbitrary element
    $$\overline{g}^n\overline{g}_j(\varepsilon) \in (G/G_0 \cdot \varepsilon) \cap \mathcal{E}_j(v,B)$$
    satisfies
    $$ d(\dec{\overline{g}^n\overline{g}_j(\varepsilon)}, \ssp{\overline{g}^n\overline{g}_j(\varepsilon)}{v}) \leq B. $$
    From Lemma $\ref{lem:decaction}$, we know $\dec{\overline{g}^n\overline{g}_j(\varepsilon)} = \phi(\overline{g}^n\overline{g}_j)(\dec{\varepsilon}) = \phi(\overline{g})^n\phi(\overline{g}_j)(\dec{\varepsilon})$. Also $\ssp{\overline{g}^n\overline{g}_j(\varepsilon)}{v} = \ssp{\varepsilon}{v}$, and so we have
    $$ d(\phi(\overline{g})^n\phi(\overline{g}_j)(\dec{\varepsilon}), \ssp{\varepsilon}{v}) \leq B. $$
    By Lemma \ref{lem:diagonalaction}, we know $H_j = \phi(\overline{G}_j)$ acts coarsely as the identity on $\ac{A_j}$ with uniform bound $2$, and so $\diam(H_j \cdot a) \leq 2$ for any vertex $a \in \ac{A_j}$. Therefore, we have
    $$ d(\phi(\overline{g})^n(\dec{\varepsilon}), \ssp{\varepsilon}{v}) \leq B + 2. $$
    Notice that we chose $\overline{g} \not \in \overline{G}_j$, so $\phi(\overline{g})$ acts loxodromically on $\ac{A_j}$. Since $\dec{\varepsilon}$ is bounded, the set of integers $n$ for which $\phi(\overline{g}^n)(\dec{\varepsilon})$ can lie inside the $(B+2)$-neighborhood of $\ssp{\varepsilon}{v}$ is contained in some finite interval $I_\varepsilon \subset \ZZ$. The width $W_\varepsilon := |I_\varepsilon|$ depends only on $B+2$ and the loxodromic constants of the action of $\phi(\overline{g})$ on $\ac{A_j}$; in particular, it is independent of $v$.

    Recall that there are only $E_j$ distinct $G/G_0$-orbits of $\alpha_j$-dual edges in $\sigma_0$. If we pick a set $X$ containing a representative of each distinct $G/G_0$-orbit of $\alpha_j$-dual edges, then we have
    \begin{equation*}
        \mathcal{E}_j(v,B) = \bigcup_{\varepsilon \in X} (G/G_0 \cdot \varepsilon) \cap \mathcal{E}_j(v,B) \subset \bigcup_{\varepsilon \in X} \bigcup_{n \in I_\varepsilon} \overline{g}^n\overline{G}_j \cdot \varepsilon.
    \end{equation*}
    We consider the image under $p_j: \sigma_0 \to \sigma_j$
    \begin{equation*}
    \begin{aligned}
        p_j \big( \mathcal{E}_j(v,B) \big) & \subset p_j \Big( \bigcup_{\varepsilon \in X} \bigcup_{n \in I_\varepsilon} \overline{g}^n\overline{G}_j \cdot \varepsilon \Big) \\
        & = \bigcup_{\varepsilon \in X} \bigcup_{n \in I_\varepsilon} p_j(\overline{g}^n\overline{G}_j \cdot \varepsilon) \\
        & = \bigcup_{\varepsilon \in X} \bigcup_{n \in I_\varepsilon} \overline{g}^n(\overline{G}_j\varepsilon) \subset \sigma_j.
    \end{aligned}
    \end{equation*}
    where $\overline{G}_j\varepsilon$ is now thought of as an edge in $\sigma_j = \sigma_0 / \overline{G}_j$. The set $\bigsqcup_{n \in I_\varepsilon} \overline{g}^n(\overline{G}_j\varepsilon)$ now has diameter at most $W_\varepsilon$ times the distance in $\sigma_j$ between $\overline{G}_j\varepsilon$ and $\overline{g}(\overline{G}_j\varepsilon)$. Take $M$ to be the maximum such diameter over all $\varepsilon \in X$, and statement (1) follows.

    The proof of (2) is nearly identical, using a $Y_k$-dual vertex $\tau$ instead of an $\alpha_j$-dual edge $\varepsilon$, $\overline{G}_{k'}$ instead of $\overline{G}_j$, and $\acc{Y_k}$ instead of $\ac{A_j}$. In fact, it is slightly simpler because any $h \in H$ restricted to $Y_k$ is either pseudo-Anosov or the identity. For $h \in H_{k'}$, the restriction cannot be pseudo-Anosov, so we obtain the uniform bound $\diam(H_{k'} \cdot a) = 0$ (rather than 2) for any vertex $a \in \acc{Y_k}$.
\end{proof}

\subsection{Bound on Image of Parallel Subtrees}
We move to bounding the image of parallel subtrees. We will need the following lemma, which is \cite[Lemma 5.6]{MR4632569}.

\begin{lemma} \label{lem:Bparallel}
    There exists a constant $B^{||} \geq 0$ such that the following holds for any simplex $u \subset \ccs{S^z}$.
    \begin{enumerate}
        \item Let $l$ be an edge path of length $2$ in $T^u \cap T^G$, $t$ be the middle vertex of $l$ and $v = \Phi(u)$. If each vertex of $l$ is of parallel-type, then $d(\dec{t}, \ssp{t}{v}) \leq B^{||}$.
        \item Let $l$ be an edge path of length $3$ in $T^u \cap T^G$, $e$ be the middle edge of $l$ and $v = \Phi(u)$. If each vertex of $e$ is of parallel-type, then $d(\dec{e}, \ssp{e}{v}) \leq B^{||}$.
    \end{enumerate}
\end{lemma}

The consequence is that away from the leaves of any parallel subtree, the distance between the decoration and subsurface projection of an edge or vertex is uniformly bounded. In practice, we can ensure any geodesic avoids the leaves simply by removing the edges on either end.

\begin{lemma} \label{lem:D0parallel}
    There exists a constant $D_0^{||} \geq 0$ such that for any simplex $u \subset \cc{S^z}$, if $T^{||}$ is a parallel subtree of $T^u \cap T^G$, then $\diam(p_0(T^{||})) \leq D_0^{||}$.
\end{lemma}
\begin{proof}
    Fix a simplex $u \subset \cc{S^z}$. Lemma \ref{lem:multilemma}(2) and (3) show that if $T^u$ has finite diameter, then it actually has diameter at most $1$, and the conclusion is trivial. We thus suppose $T^u$ has infinite diameter; in particular, $u$ only contains surviving curves. Thus, we have $u \subset \ccs{S^z}$, and we let $v = \Phi(u)$.
    
    Let $\gamma \subset T^{||}$ be a geodesic in a parallel subtree. Let $\gamma' \subset \gamma$ be the subsegment obtained by removing the edges on either end. By Lemma \ref{lem:Bparallel}, for each edge $e \subset \gamma'$, we have $d(\dec{e}, \ssp{e}{v}) \leq B^{||}$. Similarly, for each $t \in \gamma'$, we have $d(\dec{t}, \ssp{t}{v}) \leq B^{||}$. Since we are in a parallel subtree, Lemma \ref{lem:parallelstructure} guarantees for each $e \subset \gamma'$, $\ssp{e}{v} \neq \emptyset$ and for each $t \in \gamma'$, $\ssp{t}{v} \neq \emptyset$.
    
    Fix some $j \in J$. For each $\alpha_j$-dual edge $e \subset \gamma'$, we have that $e$ is twist, $\ssp{e}{v} \neq \emptyset$, and $d(\dec{e}, \ssp{e}{v}) \leq B^{||}$, meaning $e \in \widetilde{\mathcal{E}}_j(v,B^{||})$. By Corollary \ref{cor:Lbound}, every length $L_0$ subsegment of $\gamma'$ must contain an $\alpha_j$-dual edge. Combining these two facts, the $L_0$-neighborhood of $\widetilde{\mathcal{E}}_j$ covers $\gamma'$. This argument holds for all $j \in J$, so
    
    $$ \gamma' \subset \bigcap_{j \in J} N_L(\widetilde{\mathcal{E}}_j). $$
    
    Similarly, fix some $k \in K$. For each $Y_k$-dual vertex $t \in \gamma'$, we have that $t$ is pseudo-Anosov, $\ssp{t}{v} \neq \emptyset$, and $d(\dec{t}, \ssp{t}{v}) \leq B^{||}$, meaning $t \in \widetilde{\mathcal{V}}_k(v,B^{||})$. Choosing any curve $\alpha_j \subset \partial Y_k$, every length $L_0$ subsegment of $\gamma'$ must contain an $\alpha_j$-dual edge and thus a $Y_k$-dual vertex. The $L_0$-neighborhood of $\widetilde{\mathcal{V}}_k$ covers $\gamma'$. This argument holds for all $k \in K$, so
    
    $$ \gamma' \subset \bigcap_{k \in K} N_L(\widetilde{\mathcal{V}}_k). $$
    
    Combining the statements,
    $$ \gamma' \subset \bigcap_{j \in J} N_L(\widetilde{\mathcal{E}}_j) \cap \bigcap_{k \in K} N_L(\widetilde{\mathcal{V}}_k) \subset T^G. $$
    
    We first consider the image under $p_0$,
    \begin{equation*}
    \begin{aligned}
        p_0(\gamma') & \subset p_0 \Big( \bigcap_{j \in J} N_L(\widetilde{\mathcal{E}}_j) \cap \bigcap_{k \in K} N_L(\widetilde{\mathcal{V}}_k) \Big) \\
        & \subset \bigcap_{j \in J} p_0 \big( N_L(\widetilde{\mathcal{E}}_j) \big) \cap \bigcap_{k \in K} p_0 \big( N_L(\widetilde{\mathcal{V}}_k) \big) \\
        & \subset \bigcap_{j \in J} N_L(\mathcal{E}_j) \cap \bigcap_{k \in K} N_L(\mathcal{V}_k) \subset \sigma_0.
    \end{aligned}    
    \end{equation*}
    
    Finally, we consider the image under $p_{JK}$,
    \begin{equation*}
    \begin{aligned}
        p_{JK} \big( p_0(\gamma') \big) & \subset p_{JK} \Big( \bigcap_{j \in J} N_L(\mathcal{E}_j) \cap \bigcap_{k \in K} N_L(\mathcal{V}_k) \Big) \\
        & \subset \bigcap_{j \in J} p_{JK} \big( N_L(\mathcal{E}_j) \big) \cap \bigcap_{k \in K} p_{JK} \big( N_L(\mathcal{V}_k) \big) \\
        & \subset \prod_{j \in J} p_j \big( N_L(\mathcal{E}_j) \big) \times \prod_{k \in K} p_{k'} \big( N_L(\mathcal{V}_k) \big) \\
        & \subset \prod_{j \in J} N_L \big( p_j(\mathcal{E}_j) \big) \times \prod_{k \in K} N_L \big( p_k(\mathcal{V}_k) \big) \subset \sigma_{JK}.
    \end{aligned}    
    \end{equation*}
    
    By Lemma \ref{lem:pjkbounded}, $N_L \big( p_j(\mathcal{E}_j) \big)$ can have diameter at most $E_j(M+2L)$ in $\sigma_j$, and $N_L \big( p_{k'}(\mathcal{V}_k) \big)$ can have diameter at most $V_k(M+2L)$ in $\sigma_k$, so the image
    \begin{equation*}
    \begin{aligned}
        p_{JK} \big( p_0(\gamma') \big) & \subset p_{JK} \Big( \bigcap_{j \in J} N_L(\mathcal{E}_j) \cap \bigcap_{k \in K} N_L(\mathcal{V}_k) \Big) \\
        & \subset \prod_{j \in J} N_L \big( p_j(\mathcal{E}_j) \big) \times \prod_{k \in K} N_L \big( p_{k'}(\mathcal{V}_k) \big) \subset \sigma_{JK}
    \end{aligned}    
    \end{equation*}
    can have diameter at most $(E+V)(M+2L)$ in $\sigma_{JK}$. Since $p_{JK}$ is a quasi-isometry for some quasi-isometric constants $(\kappa,\lambda)$, the set
    $$ p_0(\gamma') \subset \bigcap_{j \in J} N_L(\mathcal{E}_j) \cap \bigcap_{k \in K} N_L(\mathcal{V}_k) $$
    can have diameter at most $\kappa(E+V)(M+2L) + \lambda$ in $\sigma_0$. Setting $D_0^{||} = \kappa(E+V)(M+2L) + \lambda + 2$, we have
    $$ \diam(p_0(\gamma)) \leq D_0^{||}. $$
    Since $\gamma$ was an arbitrary geodesic in $T^{||}$,
    $$ \diam(p_0(T^{||})) \leq D_0^{||}. $$
\end{proof}

\subsection{Bound on Image of Hull Subtree}
We move to bounding the image of the hull subtree. For this case, we will need the following lemma, which is \cite[Lemma 5.9]{MR4632569}. Recall the constant $B_0$ from Corollary \ref{cor:B0}.
\begin{lemma} \label{lem:Bhull}
    There exists a constant $B^\mathfrak{H} \geq B_0$ such that for any simplex $u \subset \ccs{S^z}$, any hull-type vertex $t$, and any hull-type edge $e$, we have the following. Let $v = \Phi(u)$.
    \begin{enumerate}
        \item If $d(\dec{t}, \ssp{t}{v}) > B^\mathfrak{H}$, then $t$ is a valence $1$ vertex of the hull subtree.
        \item If $d(\dec{e}, \ssp{e}{v}) > B^\mathfrak{H}$, then the hull subtree is just the single edge $e$.
    \end{enumerate}
\end{lemma}

Similar to the parallel subtree case, this lemma implies that if we avoid the leaves of the hull subtree, the decoration-subsurface projection distance is uniformly bounded. Complications arise, however, since there is no analog of Lemma \ref{lem:parallelstructure} to ensure non-empty subsurface projections. To remedy this issue, the following lemma shows that in the case of empty subsurface projections, the hull subtree geodesics are bounded.

\begin{lemma} \label{lem:Lbarbound}
    There is a constant $\overline{L} > 0$ such that for any simplex $u \subset \ccs{S^z}$, if there exists some geodesic edge path $\gamma \subset T^\hh \subset T^u \cap T^G$ with length $\geq \overline{L}$, then for each component $\alpha_j$ of $\alpha$, we have $\ssp{j}{v} \neq \emptyset$, where $v = \Phi(u)$.
\end{lemma}
\begin{proof}
    We first give the quantity $\overline{L}$, then argue the result. The construction of $\overline{L}$ follows.
    
    Given any $g \in G$, recall that $\Phi_*(g) \in H$ has a vector representation
    \[
    V(\Phi_*(g)) = \begin{bmatrix} q_1 \\ \vdots \\ q_m \\ q_{m+1} \\ \vdots \\ q_{m+l} \end{bmatrix}.
    \]
    Certain conditions on $g$ translate to numerical limitations on the entries of $V(\Phi_*(g))$ as we will explain. Even with no condition on $g$, recall that for each non-twist $\alpha_j$ and each identity component $Y_k$, we have $q_j = q_{k'} = 0$.

    Suppose that $g$ satisfies the condition that $g(e) = e'$ for some edges $e,e'$ of $T^G$ that are both $\alpha_j$-dual with $\ssp{j}{v} \neq \emptyset$ and
    \begin{equation*}
    \begin{gathered}
        d(\dec{e}, \ssp{j}{v}) \leq B^\hh, \\
        d(\dec{e'}, \ssp{j}{v}) \leq B^\hh,
    \end{gathered}
    \end{equation*}
    for $B^\hh$ from Lemma \ref{lem:Bhull}. Applying Lemma \ref{lem:decaction}, the second equation becomes
    $$ d(\Phi_*(g)(\dec{e}), \ssp{j}{v}) \leq B^\mathfrak{H}. $$
    We can decompose $\Phi_*(g) = \tau_j^{q_j} \circ h_j$, where $h_j \in H_j$. By Lemma \ref{lem:diagonalaction}, we know the action of $H_j$ on $\ac{A_j}$ moves points a distance at most $2$, so
    \begin{equation*}
    \begin{gathered}
        d((\tau_j^{q_j} \circ h_j)(\dec{e}), \ssp{j}{v}) \leq B^\mathfrak{H}, \\
        d(\tau_j^{q_j}(\dec{e}), \ssp{j}{v}) \leq B^\mathfrak{H} + 2.
    \end{gathered}
    \end{equation*}
    As in the proof of Lemma \ref{lem:pjkbounded}, since $\tau_j$ acts loxodromically on $\ac{A_j}$, it follows that $q_j$ is contained in some interval $I_e \subset \ZZ$ whose width $W_e$ depends only on $B^\mathfrak{H} + 2$ and the loxodromic constants of the action of $\tau_j$ on $\ac{A_j}$. In particular, the width $W_e$ is bounded independent of $v$. We set $W_j$ to be the maximum $W_e$ over representatives $e$ of the $E_j$ distinct $G$-orbits of $\alpha_j$-dual edges.

    Similarly, suppose that $g$ satisfies the condition that $g(t) = t'$ for some vertices $t,t$ of $T^G$ that are both $Y_k$-dual with $\ssp{k'}{v} \neq \emptyset$ and
    \begin{equation*}
    \begin{gathered}
        d(\dec{t}, \ssp{k'}{v}) \leq B^\mathfrak{H}, \\
        d(\dec{t'}, \ssp{k'}{v}) \leq B^\mathfrak{H}.
    \end{gathered}
    \end{equation*}
    By Lemma \ref{lem:decaction}, the second equation becomes
    $$ d(\Phi_*(g)(\dec{t}), \ssp{k'}{v}) \leq B^\mathfrak{H}. $$
    We decompose $\Phi_*(g) = \widehat{\psi}_k^{q_{k'}} \circ h_{k'}$, where $h_{k'} \in H_{k'}$. By Lemma \ref{lem:diagonalaction}, we know the action of $H_{k'}$ on $\acc{Y_k}$ is trivial, so
    \begin{equation*}
    \begin{gathered}
        d((\widehat{\psi}_k^{q_{k'}} \circ h_{k'})(\dec{t}), \ssp{k'}{v}) \leq B^\mathfrak{H}, \\
        d(\widehat{\psi}_k^{q_{k'}}(\dec{t}), \ssp{k'}{v}) \leq B^\mathfrak{H}.
    \end{gathered}
    \end{equation*}
    Since $\widehat{\psi}_k$ acts loxodromically on $\acc{Y_k}$, $q_{k'}$ is contained in some interval $I_t \subset \ZZ$ whose width $W_t$ depends only on $B^\mathfrak{H}$ and the loxodromic constants of the action of $\widehat{\psi}_k$ on $\acc{Y_k}$. In particular, the width $W_t$ is bounded independent of $v$. We set $W_{k'}$ to be the maximum $W_t$ over representatives $t$ of the $V_k$ distinct $G$-orbits of $Y_k$-dual vertices.

    We set $W$ to be the maximum width $W_j$ or $W_{k'}$ over all $\alpha_j$ and $Y_k$. We then set $N = W^{m+l} + 1$ and stress that this $N$ is independent of the choice of $v$. Now take $L$ to be the constant obtained from Corollary \ref{cor:Lbound}, and take $\mathfrak{L}(N,L)$ from Lemma \ref{lem:Gorbitsubsegments}. We claim that setting $\overline{L} = \mathfrak{L}(N,L) + 2$ will suffice, and remains independent of $v$.
    
    We now prove the result for $\overline{L} = \mathfrak{L}(N,L) + 2$. Suppose $\gamma \subset T^\hh$ is a geodesic edge path of length $\overline{L}$. Let $\gamma' \subset \gamma$ be the subsegment obtained by removing the edges on either end. Since the length of $\gamma'$ is at least $\mathfrak{L}(N,L)$, Lemma \ref{lem:Gorbitsubsegments} states that we can find $N+1$ disjoint subsegments $\gamma_0, \ldots, \gamma_{N} \subset \gamma'$ each of which are length $L$ and lie in the same $G$-orbit. Since each subsegment lies in the same $G$-orbit, we can find group elements $g_1, \ldots, g_N \in G$, such that $g_i(\gamma_0) = \gamma_i$ for $i \in \oneto{N}$. Since $\gamma_0$ has length $L$, by Corollary \ref{cor:Lbound} it must contain at least one edge dual to each $\alpha_j$ and by extension must also contain at least one vertex dual to each $Y_k$.

    Recall $\calS(v)$ is the set of reducing curves and subsurfaces $v$ intersects. Let $V'(h)$ be the tuple obtained from $V(h)$ by omitting all entries except the ones corresponding to curves or subsurfaces in $\calS(v)$.

    Fix an $\alpha_j \in \calS(v)$. If $\alpha_j$ is non-twist, the $\alpha_j$ entry of $V'(h)$ for any $h$ is $0$. If $\alpha_j$ is twist, let $e_0 \subset \gamma_0$ be an $\alpha_j$-dual edge. Now for each $i \in \oneto{N}$, we have $e_i = g_i(e_0) \subset \gamma_i$, and by Lemma \ref{lem:Bhull},
    \begin{equation*}
    \begin{gathered}
        d(\dec{e_0}, \ssp{j}{v}) \leq B^\mathfrak{H}, \\
        d(\dec{e_i}, \ssp{j}{v}) \leq B^\mathfrak{H}.
    \end{gathered}
    \end{equation*}
    Thus, for every $i$, the $\alpha_j$ entry of $V'(\Phi_*(g_i))$ lies in an interval of width $\leq W$.

    Similarly, fix a $Y_k \in \calS(v)$. If $Y_k$ is an identity component, the $Y_k$ entry of $V'(h)$ for any $h$ is $0$. If $Y_k$ is pseudo-Anosov, let $t_0 \in \gamma_0$ be a $Y_k$-dual edge. Now for each $i \in \oneto{N}$, we have $t_i = g_i(t_0) \in \gamma_i$, and
    \begin{equation*}
    \begin{gathered}
        d(\dec{t_0}, \ssp{k'}{v}) \leq B^\mathfrak{H}, \\
        d(\dec{t_i}, \ssp{k'}{v}) \leq B^\mathfrak{H}.
    \end{gathered}
    \end{equation*}
    Thus, for every $i$, the $Y_k$ entry of $V'(\Phi_*(g_i))$ lies in an interval of width $\leq W$.

    In summary, there are at most $W$ possible integer-values that each entry of $V'(\Phi_*(g_i))$ could take, and there are $|\calS(v)| \leq m+l$ total entries. Thus, there are at most $W^{m+l}$ possible tuple-values that each $V'(\Phi_*(g_i))$ could take. However, there are $N = W^{m+l} + 1$ total $g_i$, so the pigeonhole principle implies that there must be some $i_0 \neq i_1$ such that $V'(\Phi_*(g_{i_0})) = V'(\Phi_*(g_{i_1}))$.

    Let $g = g_{i_1}g_{i_0}^{-1}$; we have $g(\gamma_{i_0}) = \gamma_{i_1}$. Observe that
    \begin{align*}
        V'(\Phi_*(g)) & = V'(\Phi_*(g_{i_1}g_{i_0}^{-1})) \\
        & = V'(\Phi_*(g_{i_1})\Phi_*(g_{i_0}^{-1})) \\
        & = V'(\Phi_*(g_{i_1})) + V'(\Phi_*(g_{i_0}^{-1})) \\
        & = V'(\Phi_*(g_{i_1})) - V'(\Phi_*(g_{i_0})) \\
        & = 0.
    \end{align*}
    Thus, $\Phi_*(g)$ acts as the identity on each curve and subsurface in the $\calS(v)$. In particular, $p^{-1}(v)$ is $g$-invariant in $\HH^2$.
    
    Suppose for contradiction that for some component $\alpha_j \subset \alpha$, we have $\ssp{j}{v} = \emptyset$. Let $e_0$ be an $\alpha_j$-dual edge in $\gamma_0$. Then $e_{i_0} = g_{i_0}(e_0)$ and $e_{i_1} = g_{i_1}(e_0)$ are $\alpha_j$-dual edges in $\gamma_{i_0}$ and $\gamma_{i_1}$, respectively. Further, $g(e_{i_0}) = e_{i_1}$.

    We will now produce a $g$-invariant axis in $\HH^2$ disjoint from $p^{-1}(v)$. Recall that we have embedded $T^G$ into $\hh_G \subset \HH^2$, $G$-equivariantly on the vertices. Using this embedding, let $\gamma'' \subset \gamma'$ be the subsegment which starts with $e_{i_0}$ and ends with $e_{i_1}$. Let $\nu \subset \gamma''$ be the subpath starting from $e_{i_0} \cap \widetilde{\alpha}_{e_{i_0}}$ and ending at $e_{i_1} \cap \widetilde{\alpha}_{e_{i_1}}$. Note $\widetilde{\alpha}_{e_{i_0}}, \widetilde{\alpha}_{e_{i_1}}$ are both disjoint from $p^{-1}(v)$, so either we can homotope $\nu$ disjoint from $p^{-1}(v)$ or some component $\widetilde{v} \subset p^{-1}(v)$ separates $\widetilde{\alpha}_{e_{i_0}}$ from $\widetilde{\alpha}_{e_{i_1}}$. The second case cannot occur however, since then $\hh_G \cap \hh_u$ must lie on one side or the other of $\widetilde{v}$, which contradicts the fact that both $e_{i_0}$ and $e_{i_1}$ are in $T^\hh$. Thus, there is some $\nu'$ homotopic rel endpoints to $\nu$ in $\HH^2$ that is disjoint from $p^{-1}(v)$.

    Let $\nu''$ be the path obtained by concatenating $\nu'$ with an arc of $\widetilde{\alpha}_{e_{i_1}}$ from the terminal endpoint of $\nu'$ to the $g$-image of the initial endpoint. Since $\widetilde{\alpha}_{e_{i_1}}$ is disjoint from $p^{-1}(v)$, $\nu''$ remains disjoint from $p^{-1}(v)$. Now set
    $$ \widetilde{\nu} = \bigcup_{n \in \ZZ} g^n(\nu''), $$
    which is a bi-infinite, $g$-invariant path that again remains disjoint from $p^{-1}(v)$. Thus, $\widetilde{\nu}$ is contained in a single component of $\HH^2 \setminus p^{-1}(v)$ and witnesses the fact that $g$ is contained in the stabilizer of this component. The closure of this component is $\hh_{u_0}$ for some curve $u_0$ in $S^z$ with $\Phi(u_0) = v$, and therefore $g$ fixes $u_0$. Now by Theorem \ref{thm:kra}, $g$ cannot be pseudo-Anosov, so $G$ cannot purely pseudo-Anosov, a contradiction. It must have been that $\ssp{j}{v} \neq \emptyset$.
\end{proof}

We can now prove the bound on the image of the hull subtree.
\begin{lemma} \label{lem:D0hull}
    There exists a constant $D_0^{\hh} \geq 0$ such that for any simplex $u \subset \cc{S^z}$, if $T^{\hh}$ is the hull subtree of $T^u \cap T^G$, then $\diam(p_0(T^{\hh})) \leq D_0^{\hh}$.
\end{lemma}
\begin{proof}
    Fix a simplex $u \subset \cc{S^z}$. Lemma \ref{lem:multilemma}(2) and (3) show that if $T^u$ has finite diameter, then it actually has diameter at most $1$, and the conclusion is trivial. We thus suppose $T^u$ has infinite diameter; in particular, $u$ only contains surviving curves. Thus, we have $u \subset \ccs{S^z}$, and we let $v = \Phi(u)$.

    Let $\gamma \subset T^{\hh}$ be a geodesic in the hull subtree. If there is some $\alpha_j \subset \alpha$ with $\ssp{j}{v} = \emptyset$, then Lemma \ref{lem:Lbarbound} implies that the length of $\gamma$ is strictly bounded above by $\overline{L}$, so
    $$ \diam(p_0(\gamma)) < \overline{L}. $$
    We proceed with the other case, where we have $\ssp{j}{v} \neq \emptyset$ for all $\alpha_j$, which also implies $\ssp{k'}{v} \neq \emptyset$ for all $Y_k$.
    
    Let $\gamma' \subset \gamma$ be the subsegment obtained by removing the edges on either end. By Lemma \ref{lem:Bhull}, for each edge $e \subset \gamma'$, we have $d(\dec{e}, \ssp{e}{v}) \leq B^\mathfrak{H}$. Similarly, for each $t \in \gamma'$, we have $d(\dec{t}, \ssp{t}{v}) \leq B^\mathfrak{H}$.

    Fix some $j \in J$. For each $\alpha_j$-dual edge $e \subset \gamma'$, we have that $e$ is twist, $\ssp{e}{v} \neq \emptyset$, and $d(\dec{e}, \ssp{e}{v}) \leq B^\hh$, meaning $e \in \widetilde{\mathcal{E}}_j(v,B^\hh)$. By Corollary \ref{cor:Lbound}, every length $L$ subsegment of $\gamma'$ must contain an $\alpha_j$-dual edge. Combining these two facts, the $L$-neighborhood of $\widetilde{\mathcal{E}}_j$ covers $\gamma'$. This argument holds for all $j \in J$, so
    $$ \gamma' \subset \bigcap_{j \in J} N_L(\widetilde{\mathcal{E}}_j). $$
    
    Similarly, fix some $k \in K$. For each $Y_k$-dual vertex $t \in \gamma'$, we have that $t$ is pseudo-Anosov, $\ssp{t}{v} \neq \emptyset$, and $d(\dec{t}, \ssp{t}{v}) \leq B^\hh$, meaning $t \in \widetilde{\mathcal{V}}_k(v,B^\hh)$. Choosing any curve $\alpha_j \subset \partial Y_k$, every length $L$ subsegment of $\gamma'$ must contain an $\alpha_j$-dual edge and thus a $Y_k$-dual vertex. The $L$-neighborhood of $\widetilde{\mathcal{V}}_k$ covers $\gamma'$. This argument holds for all $k \in K$, so
    $$ \gamma' \subset \bigcap_{k \in K} N_L(\widetilde{\mathcal{V}}_k). $$

    From here an identical argument to the one in Lemma \ref{lem:D0parallel} gives
    $$ p_0(\gamma) \leq \kappa(E+V)(M+2L) + \lambda + 2, $$
    where $(\kappa,\lambda)$ were the quasi-isometry constants of $p_{JK}$.

    In either case, setting $D_0^{\hh} = \max \Big( \overline{L},\kappa(E+V)(M+2L) + \lambda + 2 \Big)$ gives
    $$ \diam(p_0(\gamma)) \leq D_0^{\hh}. $$

    Since $\gamma$ was an arbitrary geodesic in $T^\hh$,
    $$ \diam(p_0(T^\hh)) \leq D_0^{\hh}. $$
\end{proof}

\subsection{Proof of Proposition \ref{prop:bounddownstairs}}
We will show that the proposition holds for $D_0 = D_0^\hh + 2 D_0^{||} + 2$.

Let $\tau,\tau' \in p_0(T^u \cap T^G)$ be any two vertices. There must be some $t,t' \in T^u \cap T^G$ such that $p_0(t) = \tau$ and $p_0(t') = \tau'$. Let $\gamma \subset T^u \cap T^G$ be the geodesic edge path between $t$ and $t'$. The geodesic $\gamma$ decomposes into at most 5 segments: 1 segment contained in the hull subtree $T^\hh$, 2 segments contained (different) parallel subtrees $T^{||}_1,T^{||}_2$, and 2 edges joining the segments in $T^{||}_1,T^{||}_2$ with the segment in $T^\hh$. Now $p_0(\gamma)$ is an edge path between $\tau$ and $\tau'$, and it decomposes into at most 5 segments: 1 segment contained $p_0(T^\hh)$, 2 segments contained in $p_0(T^{||}_i)$, and 2 edges joining the segments in $p_0(T^{||}_i)$ with the segment in $p_0(T^\hh)$. It follows from Lemma \ref{lem:D0hull} that the length of the segment in $p_0(T^\hh)$ is at most $D^\hh$, and it follows from Lemma \ref{lem:D0parallel} that the lengths of the segments in $p_0(T^{||}_i)$ are at most $D^{||}$. Thus, $\diam(p_0(\gamma)) \leq D_0^\hh + 2 D_0^{||} + 2 = D_0$. Since $\tau,\tau'$ were arbitrary, we have that $\diam(p_0(T^u \cap T^G)) \leq D_0$.

\bibliographystyle{amsalpha}
\bibliography{refs}

\providecommand{\bysame}{\leavevmode\hbox to3em{\hrulefill}\thinspace}
\providecommand{\MR}{\relax\ifhmode\unskip\space\fi MR }
\providecommand{\MRhref}[2]{%
  \href{http://www.ams.org/mathscinet-getitem?mr=#1}{#2}
}
\providecommand{\href}[2]{#2}
\begin{thebibliography}{BBKL20}

\bibitem[BBKL20]{MR4069890}
Mladen Bestvina, Kenneth Bromberg, Autumn~E. Kent, and Christopher~J. Leininger, \emph{Undistorted purely pseudo-{A}nosov groups}, J. Reine Angew. Math. \textbf{760} (2020), 213--227. \MR{4069890}

\bibitem[Bes02]{MR1886668}
Mladen Bestvina, \emph{{$\Bbb R$}-trees in topology, geometry, and group theory}, Handbook of geometric topology, North-Holland, Amsterdam, 2002, pp.~55--91. \MR{1886668}

\bibitem[Bes04]{bestvinaquestions}
\bysame, \emph{Questions in geometric group theory}, 2004.

\bibitem[BH99]{MR1744486}
Martin~R. Bridson and Andr\'e Haefliger, \emph{Metric spaces of non-positive curvature}, Grundlehren der mathematischen Wissenschaften [Fundamental Principles of Mathematical Sciences], vol. 319, Springer-Verlag, Berlin, 1999. \MR{1744486}

\bibitem[Bir69]{MR0243519}
Joan~S. Birman, \emph{Mapping class groups and their relationship to braid groups}, Comm. Pure Appl. Math. \textbf{22} (1969), 213--238. \MR{243519}

\bibitem[BLM83]{MR0726319}
Joan~S. Birman, Alex Lubotzky, and John McCarthy, \emph{Abelian and solvable subgroups of the mapping class groups}, Duke Math. J. \textbf{50} (1983), no.~4, 1107--1120. \MR{726319}

\bibitem[Bra99]{MR1724853}
Noel Brady, \emph{Branched coverings of cubical complexes and subgroups of hyperbolic groups}, J. London Math. Soc. (2) \textbf{60} (1999), no.~2, 461--480. \MR{1724853}

\bibitem[CB88]{MR964685}
Andrew~J. Casson and Steven~A. Bleiler, \emph{Automorphisms of surfaces after {N}ielsen and {T}hurston}, London Mathematical Society Student Texts, vol.~9, Cambridge University Press, Cambridge, 1988. \MR{964685}

\bibitem[CL23]{chesser2023purely}
Marissa Chesser and Christopher~J. Leininger, \emph{Purely pseudo-anosov subgroups of the genus two handlebody group}, 2023.

\bibitem[DKL14]{MR3314946}
Spencer Dowdall, Richard~P. Kent, IV, and Christopher~J. Leininger, \emph{Pseudo-{A}nosov subgroups of fibered 3-manifold groups}, Groups Geom. Dyn. \textbf{8} (2014), no.~4, 1247--1282. \MR{3314946}

\bibitem[DT15]{MR3426695}
Matthew~Gentry Durham and Samuel~J. Taylor, \emph{Convex cocompactness and stability in mapping class groups}, Algebr. Geom. Topol. \textbf{15} (2015), no.~5, 2839--2859. \MR{3426695}

\bibitem[Far06]{MR2264130}
Benson Farb, \emph{Some problems on mapping class groups and moduli space}, Problems on mapping class groups and related topics, Proc. Sympos. Pure Math., vol.~74, Amer. Math. Soc., Providence, RI, 2006, pp.~11--55. \MR{2264130}

\bibitem[FM02]{MR1914566}
Benson Farb and Lee Mosher, \emph{Convex cocompact subgroups of mapping class groups}, Geom. Topol. \textbf{6} (2002), 91--152. \MR{1914566}

\bibitem[FM12]{MR2850125}
Benson Farb and Dan Margalit, \emph{A primer on mapping class groups}, Princeton Mathematical Series, vol.~49, Princeton University Press, Princeton, NJ, 2012. \MR{2850125}

\bibitem[Ham05]{hamenstadt2005word}
Ursula Hamenst\"{a}dt, \emph{Word hyperbolic extensions of surface groups}, 2005.

\bibitem[Ham07]{MR2349677}
Ursula Hamenst\"{a}dt, \emph{Geometry of the complex of curves and of {T}eichm\"{u}ller space}, Handbook of {T}eichm\"{u}ller theory. {V}ol. {I}, IRMA Lect. Math. Theor. Phys., vol.~11, Eur. Math. Soc., Z\"{u}rich, 2007, pp.~447--467. \MR{2349677}

\bibitem[HKM07]{MR2318562}
Ko~Honda, William~H. Kazez, and Gordana Mati\'c, \emph{Right-veering diffeomorphisms of compact surfaces with boundary}, Invent. Math. \textbf{169} (2007), no.~2, 427--449. \MR{2318562}

\bibitem[IMM23]{MR4526820}
Giovanni Italiano, Bruno Martelli, and Matteo Migliorini, \emph{Hyperbolic 5-manifolds that fiber over {$S^1$}}, Invent. Math. \textbf{231} (2023), no.~1, 1--38. \MR{4526820}

\bibitem[Iva92]{MR1195787}
Nikolai~V. Ivanov, \emph{Subgroups of {T}eichm\"{u}ller modular groups}, Translations of Mathematical Monographs, vol. 115, American Mathematical Society, Providence, RI, 1992, Translated from the Russian by E. J. F. Primrose and revised by the author. \MR{1195787}

\bibitem[Kee74]{MR0379833}
Linda Keen, \emph{Collars on {R}iemann surfaces}, Discontinuous groups and {R}iemann surfaces ({P}roc. {C}onf., {U}niv. {M}aryland, {C}ollege {P}ark, {M}d., 1973), Ann. of Math. Stud., vol. No. 79, Princeton Univ. Press, Princeton, NJ, 1974, pp.~263--268. \MR{379833}

\bibitem[KL07]{MR2342811}
Richard~P. Kent, IV and Christopher~J. Leininger, \emph{Subgroups of mapping class groups from the geometrical viewpoint}, In the tradition of {A}hlfors-{B}ers. {IV}, Contemp. Math., vol. 432, Amer. Math. Soc., Providence, RI, 2007, pp.~119--141. \MR{2342811}

\bibitem[KL08a]{MR2465691}
\bysame, \emph{Shadows of mapping class groups: capturing convex cocompactness}, Geom. Funct. Anal. \textbf{18} (2008), no.~4, 1270--1325. \MR{2465691}

\bibitem[KL08b]{MR2437226}
\bysame, \emph{Uniform convergence in the mapping class group}, Ergodic Theory Dynam. Systems \textbf{28} (2008), no.~4, 1177--1195. \MR{2437226}

\bibitem[KLS09]{MR2599078}
Richard~P. Kent, IV, Christopher~J. Leininger, and Saul Schleimer, \emph{Trees and mapping class groups}, J. Reine Angew. Math. \textbf{637} (2009), 1--21. \MR{2599078}

\bibitem[KMT17]{MR3695858}
Thomas Koberda, Johanna Mangahas, and Samuel~J. Taylor, \emph{The geometry of purely loxodromic subgroups of right-angled {A}rtin groups}, Trans. Amer. Math. Soc. \textbf{369} (2017), no.~11, 8179--8208. \MR{3695858}

\bibitem[Kra81]{MR0611385}
Irwin Kra, \emph{On the {N}ielsen-{T}hurston-{B}ers type of some self-maps of {R}iemann surfaces}, Acta Math. \textbf{146} (1981), no.~3-4, 231--270. \MR{611385}

\bibitem[LMS11]{MR2851869}
Christopher~J. Leininger, Mahan Mj, and Saul Schleimer, \emph{The universal {C}annon-{T}hurston map and the boundary of the curve complex}, Comment. Math. Helv. \textbf{86} (2011), no.~4, 769--816. \MR{2851869}

\bibitem[LR23]{MR4632569}
Christopher~J. Leininger and Jacob Russell, \emph{Pseudo-{A}nosov subgroups of general fibered 3-manifold groups}, Trans. Amer. Math. Soc. Ser. B \textbf{10} (2023), 1141--1172. \MR{4632569}

\bibitem[McC82]{mccarthy1982normcent}
John~D. McCarthy, \emph{Normalizers and centralizers of pseudo-anosov mapping classes}, 1982.

\bibitem[MM99]{MR1714338}
Howard~A. Masur and Yair~N. Minsky, \emph{Geometry of the complex of curves. {I}. {H}yperbolicity}, Invent. Math. \textbf{138} (1999), no.~1, 103--149. \MR{1714338}

\bibitem[MM00]{MR1791145}
H.~A. Masur and Y.~N. Minsky, \emph{Geometry of the complex of curves. {II}. {H}ierarchical structure}, Geom. Funct. Anal. \textbf{10} (2000), no.~4, 902--974. \MR{1791145}

\bibitem[MS12]{MR3000500}
Mahan Mj and Pranab Sardar, \emph{A combination theorem for metric bundles}, Geom. Funct. Anal. \textbf{22} (2012), no.~6, 1636--1707. \MR{3000500}

\bibitem[MT19]{MR3967366}
Kathryn Mann and Bena Tshishiku, \emph{Realization problems for diffeomorphism groups}, Breadth in contemporary topology, Proc. Sympos. Pure Math., vol. 102, Amer. Math. Soc., Providence, RI, 2019, pp.~131--156. \MR{3967366}

\bibitem[Run21]{MR4308279}
Ian Runnels, \emph{Effective generation of right-angled {A}rtin groups in mapping class groups}, Geom. Dedicata \textbf{214} (2021), 277--294. \MR{4308279}

\bibitem[Tsh24]{MR4684593}
Bena Tshishiku, \emph{Convex-compact subgroups of the {G}oeritz group}, Trans. Amer. Math. Soc. \textbf{377} (2024), no.~1, 271--322. \MR{4684593}

\end{thebibliography}

\end{document}